\documentclass[12pt,a4paper,reqno]{amsart} %Fr o m detta
\usepackage{amssymb}
\usepackage{latexsym}
\usepackage{amsmath}
\usepackage{mathrsfs}
\usepackage{enumerate}
\usepackage{color}
\usepackage{euscript}
\usepackage{calc}                         %Detta beh�vs f�r att fixa
\usepackage{cite}

\newcommand{\scal}[2]{\langle #1,#2\rangle}

\newcommand{\re}{\mathbf R}
\newcommand{\rr}[1]{\mathbf R^{#1}}
\newcommand{\zz}[1]{\mathbf Z^{#1}}
\newcommand{\nn}[1]{\mathbf N^{#1}}

\newcommand{\nm}[2]{\Vert #1\Vert _{#2}}
\newcommand{\nmm}[1]{\Vert #1\Vert }
\newcommand{\abp}[1]{\vert #1\vert}

\newcommand{\op}{\operatorname{Op}}

\newcommand{\sets}[2]{\{ {\,}#1{\,};{\,}#2{\,}\} }
\newcommand{\ep}{\varepsilon}
\newcommand{\fy}{\varphi}

\newcommand{\cdo}{\, \cdot \, }

\newcommand{\wpr}{{\text{\footnotesize{$\#$}}}}

\newcommand{\eabs}[1]{\langle #1\rangle}

\newcommand{\vrum}{\vspace{0.1cm}}

\newcommand{\masfR}{\mathsf R}

\newcommand{\maclM}{\mathcal M}
\newcommand{\maclS}{\mathcal S}
\newcommand{\maclV}{\mathcal V}

\newcommand{\mascF}{\mathscr F}
\newcommand{\mascS}{\mathscr S}
\newcommand{\mascP}{\mathscr P}
\newcommand{\splM}{\EuScript M}
\newcommand{\splW}{\EuScript W}
\def\rdd{\rr {2d}}

\numberwithin{equation}{section}          %Detta g�r att man f�r
\newtheorem{thm}{Theorem}
\numberwithin{thm}{section}
\newtheorem*{tom}{\rubrik}
\newcommand{\rubrik}{}
\newtheorem{prop}[thm]{Proposition}

\newtheorem{lemma}[thm]{Lemma}

\theoremstyle{definition}

\newtheorem{defn}[thm]{Definition}

\theoremstyle{remark}

\newtheorem{rem}[thm]{Remark}              %T o m hit �r bara allm�n
\title{Sharp results for the Weyl product on modulation spaces}

\author{Elena Cordero}

\address{Department of Mathematics, University of Turin, Via Carlo Alberto
10, 10123 Torino (TO), Italy.}

\email{elena.cordero@unito.it}

\author{Joachim Toft}

\address{Department of Mathematics, Linn{\ae}us University, V{\"a}xj{\"o},
Sweden}

\email{joachim.toft@lnu.se}

\author{Patrik Wahlberg}

\address{Department of Mathematics, Linn{\ae}us University, V{\"a}xj{\"o},
Sweden}

\email{patrik.wahlberg@lnu.se}


\begin{document}

\begin{abstract}
We give sufficient and necessary conditions on the Lebesgue exponents
for the Weyl product to be bounded on modulation spaces. 
The sufficient conditions are obtained as the restriction to $N=2$ of a
result valid for the $N$-fold Weyl product. As a byproduct, we obtain sharp
conditions for the twisted convolution to be bounded on Wiener
amalgam spaces.
\end{abstract}

\maketitle

%%%%%%%%%%%%%%%%%%%%
\section{Introduction}\label{sec0}
%%%%%%%%%%%%%%%%%%%%

In the paper we prove necessary and sufficient conditions for
the Weyl product to be continuous on modulation spaces, and for the
twisted convolution to be continuous on Wiener amalgam spaces.
We relax the sufficient conditions in \cite{HTW} and we prove that
the obtained conditions are also necessary.

\par

The Weyl calculus is a part of the theory of pseudo-differential operators.
For an appropriate distribution $a$ (the symbol) defined on the phase space
$T^*\rr d \simeq \rr {2d}$, the \emph{Weyl operator} $\op ^w(a)$
is a linear map between spaces of functions or distributions on $\rr d$.
(See Section \ref{sec1} for definitions.)
Weyl operators appear in various fields. 
In mathematical analysis they are used to represent linear operators,
in particular linear partial differential operators, acting between
appropriate function and distribution spaces. Weyl operators also appear in
quantum mechanics where a real-valued observable $a$ in classical
mechanics corresponds to the self-adjoint Weyl operator
$\op ^w(a)$ in quantum mechanics. For this reason
$\op ^w(a)$ is often called the Weyl quantization of $a$. In time-frequency
analysis pseudo-differential operators are used as models of non-stationary
filters.

\par

In the Weyl calculus operator composition
corresponds on the symbol level to the \emph{Weyl product},
or the \emph{twisted product}, denoted by $\wpr$.
This means that the Weyl product $a_1\wpr a_2$ of appropriate
functions or distributions $a_1$ and $a_2$ satisfies
$$
\op ^w(a_1\wpr a_2) = \op ^w(a_1)\circ \op ^w(a_2).
$$
A basic problem is to find conditions
that are necessary or sufficient for the bilinear map
\begin{equation}\label{Weylmap}
(a_1,a_2)\mapsto a_1\wpr a_2
\end{equation}
to be well-defined and continuous. Here we investigate these questions
when the factors belong to modulation spaces, a family of Banach
spaces of distributions which appear in time-frequency analysis,
harmonic analysis and Gabor analysis.

\par

The modulation spaces were introduced by Feichtinger \cite{Feichtinger2},
and their theory  was further developed and generalized by Feichtinger
and Gr{\"o}che-nig \cite{Feichtinger3,Feichtinger4,Feichtinger5,Grochenig0a}
into the theory of coorbit spaces. 

\par

The modulation space $M^{p,q}_{(\omega )}(\rr d)$, where $p,q \in
[1,\infty]$ and $\omega$ is a weight on $\rr d\times \rr d\simeq \rr {2d}$,
consists of all tempered distributions, or ultra-distributions, on $\rr d$,
whose short-time Fourier transforms have finite
$L^{p,q}_{(\omega )}(\rr {2d})$ norm. Thus the Lebesgue exponents
$p$ and $q$, and above all the weight $\omega$,
give a scale of function spaces $M^{p,q}_{(\omega )}$ with respect to
phase space concentration.
The definition of modulation spaces resembles that of Besov spaces,
and narrow embeddings between modulation and Besov spaces have
been found (cf. \cite{Grobner, HanWang, Okoudjou, Sugimoto1, Toft2,
Toft4, TW, WH}).
Depending on the assumptions on the weights, the modulation
spaces are subspaces of the tempered distributions or
ultra-distributions (cf. \cite{Cordero, Pilipovic1, Pilipovic2, Toft8, Toft9}).

\par

Since the early 1990s modulation spaces have been used
in the theory of pseudo-differential operators (cf. \cite{Tachizawa1}). 
Sj{\"o}strand \cite{Sjo1} introduced the modulation
space $M ^{\infty,1}(\rr {2d})$, which contains non-smooth functions,
as a symbol class.  He proved that $M ^{\infty,1}$ corresponds to an
algebra of $L^2$-bounded operators,

\par

Gr{\"o}chenig and Heil \cite{Grochenig0,Grochenig2} proved
that each operator with symbol in $M ^{\infty,1}$
is continuous on all modulation spaces $M^{p,q}$, $p,q \in
[1,\infty]$. This extends Sj{\"o}strand's $L^2$-continuity result since
$M^{2,2}=L^2$. Some generalizations to operators with
symbols in unweighted modulation spaces were obtained in
\cite{Grochenig1b,Toft2}, and \cite{Toft3, Toft5, Toft9} 
contain extensions to weighted modulation spaces.

\par

Concerning the algebraic properties of the Weyl calculus
(cf. \cite{Folland,GraciaVarilly,Ho3}) with respect to modulation
spaces, Sj\"ostrand's results \cite{Sjo1,Sjo2} were refined in \cite{Toft0}, and 
new results were found by Labate \cite{Labate1}, Gr\"ochenig and
Rzeszotnik \cite{GrochenigRzeszotnik}, and by Holst and two of the authors
\cite{HTW}. 

\par

Our main result in this paper is a multi-linear version of a generalization of
\cite[Theorem 0.3$'$]{HTW} which concerns sufficient conditions
for continuity of the Weyl product on modulation spaces.
We also prove that the sufficient conditions are necessary in the bilinear
case, for a certain family of weight functions.

\par

The Weyl product \eqref{Weylmap} is continuous $\maclS _s(\rr {2d}) \times \maclS
_s(\rr {2d}) \mapsto \maclS _s(\rr {2d})$, where $\maclS _s(\rr {2d})$ denotes the Gelfand--Shilov space of order $s$, 
for every $s\ge 0$. 
In order to explain our extension of this result to modulation spaces,
we introduce the H{\"o}lder--Young exponent function
\begin{equation}\label{HYfunctional}
\masfR _2(\mathbf p) =\sum _{j=0}^2\frac 1{p_j} -1,\quad
\mathbf p = (p_0,p_1,p_2)\in [1,\infty ]^{3}
\end{equation}
and consider weights $\omega _j$, $j=0,1,2$, in
$\mascP _E(\rr {4d})$, the set of moderate weights on $\rr  {4d}$.
We suppose that
\begin{multline}\label{weightcond}
C\le \omega _0(X_2+X_0,X_2-X_0)\prod _{j=1}^2
\omega _j(X_j+X_{j-1},X_j-X_{j-1}),
\\[1ex]
X_0,X_1,X_2\in \rr {2d},
\end{multline}
holds for some $C>0$.

\par

With these terms our result in the bilinear case on sufficient conditions
for continuity of the Weyl product reads as follows. Here
$\splM ^{p,q}_{(\omega )}(\rr {2d})$, as opposed to
$M ^{p,q}_{(\omega )}(\rr {2d})$, is the modulation space
defined with the symplectic Fourier transform instead of the
usual Fourier transform. 

\par

\begin{thm}\label{thm0.3}
Let $p_j,q_j\in [1,\infty ]$, $j=0,1,2$, and suppose
\begin{equation}\label{pqconditions}
\max \left ( \masfR _2(\mathbf q') ,0 \right )
\le  \min _{j=0,1,2}
\left ( \frac {1}{p_j},\frac {1}{q_j'}, \masfR _2(\mathbf p)\right ).
\end{equation}
Let $\omega _j\in \mascP _E(\rr {4d})$, $j=0,1,2$, and suppose
\eqref{weightcond} holds. Then the map
\eqref{Weylmap} from $\maclS _{1/2}(\rr {2d}) \times \maclS _{1/2}(\rr {2d})$
to $\maclS _{1/2}(\rr {2d})$
extends uniquely to a continuous map from $\splM ^{p_1,q_1}
_{(\omega _1)}(\rr {2d}) \times \splM ^{p_2,q_2}_{(\omega _2)}(\rr {2d})$
to $\splM ^{p_0',q_0'} _{(1/\omega _0)}(\rr {2d})$.
\end{thm}

\par

This result is the restriction to $N=2$ of a multi-linear result treating
the Weyl product of $N$ factors $a_1 \wpr \dots \wpr a_N$ proved
in Section \ref{sec2} (see Theorem  \ref{thm0.3}$'$). 
Theorem \ref{thm0.3} extends all results in the literature, familiar to us,
on the Weyl product acting on modulation spaces, in particular
\cite[Theorem 0.3$'$]{HTW} and its slight extension
\cite[Theorem 6.4]{Toft8}. In Section \ref{sec4} we present a table
which explains the difference between \cite[Theorem 0.3$'$]{HTW}
and Theorem  \ref{thm0.3} in the important cases when the
Lebesgue exponents $p_j,q_j$ belong to $\{ 1,2,\infty \}$.

\par

In Section \ref{sec2} we also present a parallel result to
Theorem  \ref{thm0.3}$'$ on sufficient conditions for
continuity of the Weyl product on modulation spaces. 
It gives continuity in certain cases not covered by Theorem
\ref{thm0.3}$'$ with $N>2$,  e.{\,}g. when several of the Weyl operators are
Hilbert--Schmidt operators (cf. Theorem \ref{thm0.4}).
Section \ref{sec2} ends with a continuity result for the
twisted convolution on Wiener amalgam spaces 
(cf. Theorem \ref{thm0.3TwistConv}). 

\par

In Section \ref{sec3} we prove that Theorem \ref{thm0.3} is sharp
with respect to the conditions on the Lebesgue exponents
$p_j$ and $q_j$,
for triplets $(\omega_0,\omega_1,\omega_2)$ of
polynomially moderate weights that are interrelated in a certain way
(see \eqref{weightcond2})
which implies that \eqref{weightcond} is automatically satisfied. 
The sharpness means that
\eqref{pqconditions} must hold when the map \eqref{Weylmap} from 
$\maclS _{1/2} \times \maclS _{1/2}$ to $\maclS _{1/2}$ is extendable
to a continuous map from
$\splM ^{p_1,q_1}_{(\omega _1)} \times \splM ^{p_2,q_2}
_{(\omega _2)}$ to $\splM ^{p_0',q_0'}_{(1/\omega _0)}$ (cf. Theorem
\ref{thm0.3Rev}).

\par

%%%%%%%%%%%%%%%%%%%%%%%%%%
\section{Preliminaries}\label{sec1}
%%%%%%%%%%%%%%%%%%%%%%%%%%

\par

In this section we introduce notation and discuss the background on 
Gelfand--Shilov spaces, 
pseudo-differential operators, the Weyl
product, twisted convolution and modulation spaces.
Most proofs can be found in the literature and are therefore omitted. 

\par

Let $0<h,s\in \mathbf R$ be fixed. The space $\mathcal S_{s,h}(\rr d)$
consists of all $f\in C^\infty (\rr d)$ such that
\begin{equation*}
\nm f{\mathcal S_{s,h}}\equiv \sup \frac {|x^\beta \partial ^\alpha
f(x)|}{h^{|\alpha | + |\beta |}\alpha !^s\, \beta !^s}
\end{equation*}
is finite, with supremum taken over all $\alpha ,\beta \in
\mathbf N^d$ and $x\in \rr d$.

\par

The space $\mathcal S_{s,h} \subseteq
\mathscr S$ ($\mathscr S$ denotes the Schwartz space) is a Banach space
which increases with $h$ and $s$.
Inclusions between topological spaces are understood to be continuous. 
If $s>1/2$, or $s =1/2$ and $h$ is sufficiently large, then $\mathcal
S_{s,h}$ contains all finite linear combinations of Hermite functions.
Since the space of such linear combinations is dense in $\mathscr S$, it follows
that the topological dual $(\mathcal S_{s,h})'(\rr d)$ of $\mathcal S_{s,h}(\rr d)$ is
a Banach space which contains $\mathscr S'(\rr d)$.

\par

The \emph{Gelfand--Shilov spaces} $\mathcal S_{s}(\rr d)$ and
$\Sigma _{s}(\rr d)$ (cf. \cite{GS}) are the inductive and projective limits, respectively,
of $\mathcal S_{s,h}(\rr d)$, with respect to the parameter $h$. Thus
\begin{equation}\label{GSspacecond1}
\mathcal S_{s}(\rr d) = \bigcup _{h>0}\mathcal S_{s,h}(\rr d)
\quad \text{and}\quad \Sigma _{s}(\rr d) =\bigcap _{h>0}\mathcal S_{s,h}(\rr d),
\end{equation}
where $\mathcal S_{s}(\rr d)$ is equipped with the the strongest topology such
that the inclusion map from $\mathcal S_{s,h}(\rr d)$ into $\mathcal S_{s}(\rr d)$
is continuous, for every choice of $h>0$. The space $\Sigma _s(\rr d)$ is a
Fr{\'e}chet space with seminorms
$\nm \cdo{\mathcal S_{s,h}}$, $h>0$. We have $\Sigma _s(\rr d)\neq \{ 0\}$
if and only if $s>1/2$, and $\maclS _s(\rr d)\neq \{ 0\}$
if and only if $s\ge 1/2$. From now on we assume that $s>1/2$ when we
consider $\Sigma _s(\rr d)$, and $s\ge 1/2$ when we consider $\maclS
_s(\rr d)$.

\medspace

The \emph{Gelfand--Shilov distribution spaces} $\mathcal S_{s}'(\rr d)$
and $\Sigma _s'(\rr d)$ are the projective and inductive limit
respectively of $\mathcal S_s'(\rr d)$.  This means that
\begin{equation}\tag*{(\ref{GSspacecond1})$'$}
\mathcal S_s'(\rr d) = \bigcap _{h>0}\mathcal S_{s,h}'(\rr d)\quad
\text{and}\quad \Sigma _s'(\rr d) =\bigcup _{h>0} \mathcal S_{s,h}'(\rr d).
\end{equation}
In \cite{GS, Ko, Pil} it is proved that $\mathcal S_s'(\rr d)$
is the topological dual of $\mathcal S_s(\rr d)$, and $\Sigma _s'(\rr d)$
is the topological dual of $\Sigma _s(\rr d)$.

\par

For each $\ep >0$ and $s>1/2$ we have
\begin{equation}\label{GSembeddings}
\begin{alignedat}{3}
\maclS _{1/2}(\rr d) & \subseteq &\Sigma _s (\rr d) & \subseteq&
\maclS _s(\rr d) & \subseteq \Sigma _{s+\ep}(\rr d)
\\[1ex]
\quad \text{and}\quad
\Sigma _{s+\ep}' (\rr d) & \subseteq  & \maclS _s'(\rr d)
& \subseteq & \Sigma _s'(\rr d) & \subseteq \maclS _{1/2}'(\rr d).
\end{alignedat}
\end{equation}

\par

The Gelfand--Shilov spaces are invariant under several basic operations,
e.{\,}g. translations, dilations, tensor products
and (partial) Fourier transformation.

\par

We normalize the Fourier transform of $f\in L^1(\rr d)$ as
$$
(\mathscr Ff)(\xi )= \widehat f(\xi ) \equiv (2\pi )^{-d/2}\int _{\rr
{d}} f(x)e^{-i\scal  x\xi }\, dx,
$$
where $\scal \cdo \cdo$ denotes the scalar
product on $\rr d$. The map $\mathscr F$ extends
uniquely to homeomorphisms on $\mathscr S'(\rr d)$, $\mathcal S_s'(\rr d)$
and $\Sigma _s'(\rr d)$, and restricts to
homeomorphisms on $\mathscr S(\rr d)$, $\mathcal S_s(\rr d)$ and
$\Sigma _s(\rr d)$, and to a unitary operator on $L^2(\rr d)$.

\medspace

Next we recall some basic facts from pseudo-differential calculus
(cf. \cite{Ho3}). Let $s\ge 1/2$, $a\in \maclS _s
(\rr {2d})$, and $t\in \mathbf R$ be fixed. The
pseudo-differential operator $\op _t(a)$ defined by
\begin{equation}\label{e0.5}
\op _t(a)f(x) = (2\pi )^{-d}\iint _{\rr {2d}}a((1-t)x+ty,\xi )f(y)e^{i\scal {x-y}\xi}\,
dyd\xi
\end{equation}
is a linear and continuous operator on $\maclS _s (\rr d)$.
For $a\in \maclS _s'(\rr {2d})$ the
pseudo-differential operator $\op _t(a)$ is defined as the continuous
operator from $\maclS _s(\rr d)$ to $\maclS _s'(\rr d)$ with
distribution kernel given by
\begin{equation}\label{atkernel}
K_{a,t}(x,y)=(2\pi )^{-d/2}(\mascF _2^{-1}a)((1-t)x+ty,x-y).
\end{equation}
Here $\mascF _2F$ is the partial Fourier transform of $F(x,y)\in
\maclS _s'(\rr {2d})$ with respect to the variable $y \in \rr d$. This
definition generalizes \eqref{e0.5} and is well defined, since the mappings
\begin{equation}\label{homeoF2tmap}
\mascF _2\quad \text{and}\quad F(x,y)\mapsto F((1-t)x+ty,y-x)
\end{equation}
are homeomorphisms on $\maclS _s'(\rr {2d})$.
The map $a\mapsto K_{a,t}$ is hence a homeomorphism on
$\maclS _s'(\rr {2d})$.

\par

For any $K\in \maclS '_s(\rr {d_1+d_2})$, let $T_K$ be the
linear and continuous mapping from $\maclS _s(\rr {d_1})$
to $\maclS _s'(\rr {d_2})$ defined by
\begin{equation}\label{pre(A.1)}
(T_Kf,g)_{L^2(\rr {d_2})} = (K,g\otimes \overline f )_{L^2(\rr {d_1+d_2})},
\quad f \in \maclS _s(\rr {d_1}), \quad g \in \maclS _s(\rr {d_2}).
\end{equation}
It is a well known consequence of the Schwartz kernel
theorem that if $t\in \mathbf R$, then $K\mapsto T_K$ and $a\mapsto
\op _t(a)$ are bijective mappings from $\mascS '(\rr {2d})$
to the space of linear and continuous mappings from $\mascS (\rr d)$ to
$\mascS '(\rr d)$ (cf. e.{\,}g. \cite{Ho3}).

\par

Likewise the maps $K\mapsto T_K$
and $a\mapsto \op _t(a)$ are uniquely extendable to bijective
mappings from $\maclS _s'(\rr {2d})$ to the set of linear and
continuous mappings from $\maclS _s(\rr d)$ to $\maclS _s'(\rr d)$.
In fact, the asserted bijectivity for the map $K\mapsto T_K$ follows from
the kernel theorem \cite[Theorem 2.3]{LozPer} (cf. \cite[vol. IV]{GS}).
This kernel theorem corresponds to the Schwartz kernel theorem
in the usual distribution theory.
The other assertion follows from the fact that $a\mapsto K_{a,t}$
is a homeomorphism on $\maclS _s'(\rr {2d})$.

\par

In particular, for each $a_1\in \maclS _s '(\rr {2d})$ and $t_1,t_2\in
\mathbf R$, there is a unique $a_2\in \maclS _s '(\rr {2d})$ such that
$\op _{t_1}(a_1) = \op _{t_2} (a_2)$. The relation between $a_1$ and $a_2$
is given by
\begin{equation}\label{calculitransform}
\op _{t_1}(a_1) = \op _{t_2}(a_2) \qquad \Longleftrightarrow \qquad
a_2(x,\xi )=e^{i(t_1-t_2)\scal {D_x }{D_\xi}}a_1(x,\xi ).
\end{equation}
(Cf. \cite{Ho3}.) Note that the right-hand side makes sense, since
it means $\widehat a_2(x,\xi)=e^{i(t_1-t_2)\scal x \xi }\widehat a_1(x,\xi)$,
and since the map $a(x,\xi )\mapsto e^{it\scal x\xi }a(x,\xi )$ is continuous on
$\maclS _s '(\rr {2d})$.

\medspace

Next we discuss the Weyl product, twisted convolution and related
operations (see \cite{Ho3,Folland}). Let $s\ge 1/2$ and let $a,b\in
\maclS _s '(\rr {2d})$. The Weyl product $a\wpr
b$ between $a$
and $b$ is the function or distribution which satisfies $\op ^w(a\wpr
b) = \op ^w(a)\circ \op ^w(b)$, provided the right-hand side
makes sense as a continuous operator from $\maclS _s(\rr d)$ to
$\maclS _s'(\rr d)$. 

\par

The Wigner distribution is defined by
\begin{equation*}
W_{f,g}(x,\xi) = \mathscr F( f(x+\cdo /2) \overline{g(x-\cdo /2)}
)(\xi ), \quad f,g \in \maclS _{1/2}'(\rr d), 
\end{equation*}
and takes the form
\begin{equation*}
W_{f,g}(x,\xi) = (2\pi )^{-d/2} \int_{\rr d} f(x+y/2) \overline{g(x-y/2)}
e^{- i \scal y \xi} \, dy, 
\end{equation*}
when 
$f,g \in \maclS _{1/2} (\rr d)$. 
The Wigner distribution
appears in the Weyl calculus in the formula
\begin{align*}
(\op^w (a) f,g)_{L^2(\rr d)} & = (2\pi )^{-d/2} (a,W_{g,f} )_{L^2(\rr {2d})},
\\[1ex]
& \qquad a \in  \maclS _{1/2} '(\rr {2d}), \quad f,g \in \maclS _{1/2} (\rr d). 
\end{align*}

\par

The Weyl product can be expressed in terms
of the symplectic Fourier transform and the twisted convolution. 
The \emph{symplectic Fourier transform} of $a \in
\maclS _s (\rr {2d})$, where $s\ge 1/2$, is defined by 
\begin{equation*}
(\mascF _\sigma a) (X)
= \pi^{-d}\int _{\rr {2d}} a(Y) e^{2 i \sigma(X,Y)}\,  dY,
\end{equation*}
where $\sigma$ is the symplectic form
$$
\sigma(X,Y) = \scal y \xi -
\scal x \eta ,\qquad X=(x,\xi )\in \rr {2d},\ Y=(y,\eta )\in \rr {2d}.
$$

\par

We note that $\mascF _\sigma = T\circ (\mascF^{-1} \otimes \mascF )$, when
$(Ta)(x,\xi) = 2^d a(2\xi ,2x)$. 
The symplectic Fourier transform ${\mascF _\sigma}$ is continuous on
$\maclS _s (\rr {2d})$ and extends uniquely to a homeomorphism on
$\maclS _s '(\rr {2d})$, and to a unitary map on $L^2(\rr {2d})$, since similar
facts hold for $\mascF$. Furthermore $\mascF _\sigma^{2}$ is the identity
operator.

\par

Let $s\ge 1/2$ and $a,b\in \maclS _s (\rr {2d})$. The \emph{twisted
convolution} of $a$ and $b$ is defined by
\begin{equation}\label{twist1}
(a \ast _\sigma b) (X)
= (2/\pi)^{d/2} \int _{\rr {2d}}a(X-Y) b(Y) e^{2 i \sigma(X,Y)}\, dY.
\end{equation}
The definition of $*_\sigma$ extends in different ways. For example
it extends to a continuous multiplication on $L^p(\rr {2d})$ when $p\in
[1,2]$, and to a continuous map from $\maclS _s '(\rr {2d})\times
\maclS _s (\rr {2d})$ to $\maclS _s '(\rr {2d})$. If $a,b \in
\maclS _s '(\rr {2d})$, then $a \wpr b$ makes sense if and only if $a
*_\sigma \widehat b$ makes sense, and
\begin{equation}\label{tvist1}
a \wpr b = (2\pi)^{-d/2} a \ast_\sigma (\mascF _\sigma {b}).
\end{equation}
For the twisted convolution we have
\begin{equation}\label{weylfourier1}
\mascF _\sigma (a *_\sigma b) = (\mascF _\sigma a) *_\sigma b =
\check{a} *_\sigma (\mascF _\sigma b),
\end{equation}
where $\check{a}(X)=a(-X)$ (cf. \cite{Toft1}). A
combination of \eqref{tvist1} and \eqref{weylfourier1} gives
\begin{equation}\label{weyltwist2}
\mascF _\sigma (a\wpr b) = (2\pi )^{-d/2}(\mascF _\sigma
a)*_\sigma (\mascF _\sigma b).
\end{equation}
If $\widetilde a(X) = \overline{a(-X)}$ then
\begin{equation*}
(a_1*_\sigma a_2,b) = (a_1,b*_\sigma \widetilde a_2)=(a_2,\widetilde
a_1*_\sigma b),\quad (a_1*_\sigma a_2)*_\sigma b = a_1*_\sigma
(a_2*_\sigma b),
\end{equation*}
for appropriate $a_1, a_2,b$, and furthermore (cf. \cite{HTW})
\begin{equation}\label{duality0}
(a_1 \wpr a_2, b) = (a_2, \overline{a_1} \wpr b) = (a_1, b \wpr
\overline{a_2}).
\end{equation}

\medspace

Next we turn to the basic properties of modulation spaces, and start by
recalling the conditions for the involved weight functions. Let
$0<\omega, v\in L^\infty _{\rm loc}(\rr d)$. Then $\omega$ is called
\emph{moderate} or \emph{$v$-moderate} if
\begin{equation}\label{moderate}
\omega (x+y) \lesssim \omega (x)v(y),\quad x,y\in \rr d.
\end{equation}
Here the notation $f(x) \lesssim g(x)$ means that there exists $C>0$
such that $f(x) \leq C g(x)$ for all arguments $x$ in the domain of $f$
and $g$. If $f \lesssim g$ and $g \lesssim f$ we write $f \asymp g$. 
The function $v$ is called \emph{submultiplicative}
if it is even and \eqref{moderate} holds when $\omega =v$. We note that if
\eqref{moderate} holds then
$$
v^{-1}\lesssim \omega  \lesssim v.
$$
For such $\omega$ it follows that \eqref{moderate} is true when
$$
v(x) =Ce^{c|x|},
$$
for some positive constants $c$ and $C$ (cf. \cite{Grochenig5}).
In particular, if $\omega$ is moderate on $\rr d$, then
$$
e^{-c|x|}\lesssim \omega (x)\lesssim e^{c|x|},
$$
for some $c>0$.

\par

The set of all moderate functions on $\rr d$
is denoted by $\mascP _E(\rr d)$.
If $v$ in \eqref{moderate}
can be chosen as  $v(x)=\eabs{x}^s=(1+|x|^2)^{s/2}$ for some
$s \ge 0$, then $\omega$ is
said to be of polynomial type or polynomially moderate. We let
$\mascP (\rr d)$ be the set
of all polynomially moderate functions on $\rr d$.

\medspace

Let $\phi \in \maclS _s (\rr d) \setminus 0$ be
fixed. The \emph{short-time Fourier transform} (STFT) $V_\phi f$ of $f\in
\maclS _s ' (\rr d)$ with respect to the \emph{window function} $\phi$ is
the Gelfand--Shilov distribution on $\rr {2d}$ defined by
$$
V_\phi f(x,\xi ) \equiv 
\mascF (f \, \overline {\phi (\cdo -x)})(\xi ).
$$

\par

For $a \in \maclS _{1/2} '(\rr {2d})$ and
$\Phi \in \maclS _{1/2} (\rr {2d})  \setminus 0$ the
\emph{symplectic short-time Fourier transform} $\maclV _{\Phi} a$
of $a$ with respect to $\Phi$ is the defined similarly as
\begin{equation}\nonumber
\maclV _{\Phi} a(X,Y) = \mascF _\sigma \big( a\, \overline{ \Phi (\cdo -X) }
\big) (Y),\quad X,Y \in \rr {2d}.
\end{equation}
We have
\begin{multline}\label{stftcompare}
\maclV _{\Phi} a(X,Y) = 2^d V_\Phi a(x,\xi , -2\eta ,2y),
\\[1ex]
X=(x,\xi )\in \rr {2d},\ Y=(y,\eta )\in \rr {2d},
\end{multline}
which shows the close connection between $V_\Phi a$
and $\maclV _{\Phi} a$. 
The Wigner distribution $W_{f,\phi}$ and $V_\phi f$ are also closely related.

\par

If $f ,\phi \in \maclS _s (\rr d)$ and $a,\Phi \in \maclS _s (\rr {2d})$ then
\begin{align*}
V_\phi f(x,\xi ) &= (2\pi )^{-d/2}\int f(y)\overline {\phi
(y-x)}e^{-i\scal y\xi}\, dy 
\intertext{and}
\maclV _\Phi a(X,Y ) &= \pi ^{-d}\int a(Z)\overline {\Phi
(Z-X)}e^{2i\sigma (Y,Z)}\, dZ .
\end{align*}

\par

Let $\omega \in \mascP _E (\rr {2d})$, $p,q\in [1,\infty ]$
and $\phi \in \maclS _{1/2} (\rr d)\setminus 0$ be fixed. The
\emph{modulation space} $M^{p,q}_{(\omega )}(\rr d)$ consists of
all $f\in \maclS _{1/2} '(\rr d)$ such that
\begin{align}
\nm f{M^{p,q}_{(\omega )}} &\equiv \Big (\int _{\rr d}\Big (\int _{\rr d}
|V_\phi f(x,\xi )\omega (x,\xi )|^p\, dx\Big )^{q/p}\, d\xi \Big )^{1/q}
\label{modnorm1}
\intertext{is finite, and the Wiener amalgam
space $W^{p,q}_{(\omega )} (\rr d)$ consists of all $f\in
\maclS _{1/2} '(\rr d)$ such that}
\nm f {W^{p,q}_{(\omega )}} &\equiv \Big (\int _{\rr d}\Big (\int _{\rr d}
|V_\phi f(x,\xi )\omega (x,\xi )|^q\, d\xi \Big )^{p/q}\, dx \Big )^{1/p}
\label{modnorm2}
\end{align}
is finite (with obvious modifications in \eqref{modnorm1} and
\eqref{modnorm2} when $p=\infty$ or $q=\infty$).

\par

\begin{rem}\label{MoreWeightClasses}
As follows from Proposition \ref{p1.4} (2) below we have that in fact
$M^{p,q}_{(\omega )}(\rr d)$ contains the superspace $\Sigma _1(\rr d)$
of $\maclS _{1/2}(\rr d)$, and is contained in the subspace $\Sigma _1'(\rr d)$
of $\maclS _{1/2}'(\rr d)$, when $\omega \in \mascP _E(\rr {2d})$. Hence we
could from the beginning have assumed that $f \in \Sigma _1'(\rr d)$ in
\eqref{modnorm1} and \eqref{modnorm2}.

\par

On the other hand, in \cite{Toft8}, certain weight classes containing $\mascP
_E(\rr {2d})$ and superexponential weights are introduced. For any $s>1/2$,
the corresponding families of modulation spaces are large enough to contain
superspaces of $\maclS _s'(\rr d)$ and subspaces of $\maclS _s(\rr d)$.

\par

However, we are not dealing with these large families of modulation spaces
because we need (1) and (2) in Proposition \ref{p1.4} which are not known to be
true for weights of this generality.
\end{rem}

\par

\begin{rem}
The literature contains slightly different conventions
concerning modulation and Wiener amalgam spaces. Sometimes our
definition of a Wiener amalgam space is considered as a particular
case of a general class of modulation spaces (cf.
\cite{Feichtinger1,Feichtinger2,Feichtinger6}). Our definition is
adapted to give the relation \eqref{twistfourmod} that suits our
purpose to transfer continuity for the Weyl product on
modulation spaces to continuity for twisted convolution on Wiener
amalgam spaces. 
\end{rem}

\par

On the even-dimensional phase space $\rr {2d}$ we may define
modulation spaces based on the symplectic STFT. 
Thus if $\omega \in \mascP _E (\rr {4d})$, $p,q\in [1,\infty ]$
and $\Phi \in \maclS _{1/2} (\rr {2d})\setminus 0$ are fixed, then
the \emph{symplectic modulation spaces}
$\splM ^{p,q}_{(\omega )}(\rr {2d})$ and Wiener amalgam spaces $\splW ^{p,q}
_{(\omega )}(\rr {2d})$ are obtained by replacing
the STFT $a\mapsto V_\Phi a$ by the corresponding
symplectic version $a\mapsto \maclV _\Phi a$ in \eqref{modnorm1} and
\eqref{modnorm2}.
(Sometimes the word \emph{symplectic} before modulation space is
omitted for brevity.) 
By \eqref{stftcompare} we have
$$
\splM ^{p,q}_{(\omega )}(\rr {2d}) = M ^{p,q}_{(\omega _0)}(\rr {2d}),
\quad \omega(x,\xi, y, \eta) = \omega_0 (x,\xi, -2 \eta, 2 y).
$$
It follows that all properties which are valid for $M ^{p,q}_{(\omega )}$
carry over to $\splM ^{p,q}_{(\omega )}$.

\par

From
\begin{equation}\label{FourSTFTs}
V_{\widehat \phi}\widehat f (\xi ,-x) =e^{i\scal x\xi}V_{\phi}f(x,\xi )
\end{equation}
it follows that
$$
f\in W^{q,p}_{(\omega )}(\rr d)\quad \Longleftrightarrow \quad
\widehat f\in M^{p,q}_{(\omega _0)}(\rr d),\qquad \omega _0(\xi
,-x)=\omega (x,\xi ).
$$
In the symplectic situation these formulas read
\begin{equation}\label{stftsymplfour}
\maclV _{\mascF _\sigma \Phi}(\mascF _\sigma a)(X,Y) =
e^{2i\sigma (Y,X)}\maclV _\Phi a(Y,X)
\end{equation}
and
\begin{equation}\label{twistfourmod}
\mascF _\sigma \splM ^{p,q}_{(\omega )}(\rr {2d}) = \splW
^{q,p}_{(\omega _0)}(\rr {2d}), \qquad \omega _0(X,Y)=\omega (Y,X).
\end{equation}

\par

For brevity we denote $\splM ^p _{(\omega )}= \splM ^{p,p}_{(\omega
)}$, $\splW ^p_{(\omega )}=\splW ^{p,p}_{(\omega
)}$, and when $\omega \equiv 1$ we write $\splM ^{p,q}=\splM
^{p,q}_{(\omega )}$ and $\splW ^{p,q}=\splW ^{p,q}_{(\omega )}$. We
also let $\mathcal M^{p,q} _{(\omega )} (\rr {2d})$ be the
completion of $\maclS _s(\rr {2d})$ with
respect to the norm $\nm \cdo {\splM ^{p,q}_{(\omega )}}$.

\par

In the following proposition we list some basic facts on invariance, growth
and duality for modulation spaces. Recall that $p,p'\in[1,\infty]$ satisfy
$1/p+1/p'=1$. 
Since our main results are formulated in terms of
symplectic modulation spaces, we state the result for them
instead of the modulation spaces $M ^{p,q}_{(\omega )}(\rr {d})$.

\par

\begin{prop}\label{p1.4}
Let $p,q,p_j,q_j\in [1,\infty ]$ for $j=1,2$, and $\omega
,\omega _1,\omega _2,v\in \mascP _E (\rr {4d})$ be such that
$v=\check v$, $\omega$ is $v$-moderate and $\omega _2\lesssim
\omega _1$. Then the following is true:
\begin{enumerate}
\item[{\rm{(1)}}] $a\in \splM ^{p,q}_{(\omega )}(\rr {2d})$ if and only if
\eqref{modnorm1} holds for any $\phi \in \splM ^1_{(v)}(\rr {2d})\setminus
0$. Moreover, $\splM ^{p,q}_{(\omega )}$ is a Banach space under the norm
in \eqref{modnorm1} and different choices of $\phi$ give rise to
equivalent norms;

\vrum

\item[{\rm{(2)}}] if  $p_1\le p_2$ and $q_1\le q_2$  then
$$
\Sigma _1 (\rr {2d})\subseteq \splM ^{p_1,q_1}_{(\omega _1)}(\rr
{2d})\subseteq \splM ^{p_2,q_2}_{(\omega _2)}(\rr {2d})\subseteq
\Sigma _1 '(\rr {2d}).
$$

\vrum

\item[{\rm{(3)}}] the $L^2$ inner product $( \cdo ,\cdo )_{L^2}$ on $\maclS _{1/2}$
extends uniquely to a continuous sesquilinear form $( \cdo ,\cdo )$
on $\splM ^{p,q}_{(\omega )}(\rr {2d})\times \splM ^{p'\! ,q'}_{(1/\omega )}(\rr {2d})$.
On the other hand, if $\nmm a = \sup \abp {(a,b)}$, where the supremum is
taken over all $b\in \maclS _{1/2} (\rr {2d})$ such that
$\nm b{\splM ^{p',q'}_{(1/\omega )}}\le 1$, then $\nmm {\cdot}$ and $\nm
\cdot {\splM ^{p,q}_{(\omega )}}$ are equivalent norms;

\vrum

\item[{\rm{(4)}}] if $p,q<\infty$, then $\maclS _{1/2} (\rr {2d})$ is dense in
$\splM ^{p,q}_{(\omega )}(\rr {2d})$ and the dual space of $\splM
^{p,q}_{(\omega )}(\rr {2d})$ can be identified
with $\splM ^{p'\! ,q'}_{(1/\omega )}(\rr {2d})$, through the form
$(\cdo  ,\cdo )$. Moreover, $\maclS _{1/2} (\rr {2d})$ is weakly dense
in $\splM ^{p' ,q'}_{(\omega )}(\rr {2d})$ with respect to the form $(\cdo  ,\cdo )$
provided $(p,q) \neq (\infty,1)$ and $(p,q) \neq (1,\infty)$;

\vrum

\item[{\rm{(5)}}] if $p,q,r,s,u,v \in [1,\infty]$, $0\le \theta \le 1$,
\begin{equation*}
\frac1p = \frac {1-\theta }{r}+\frac {\theta}{u} \quad
\text{and} \quad \frac 1q = \frac {1-\theta }{s}+\frac {\theta}{v},
\end{equation*}
then complex interpolation gives
\begin{equation*}
(\maclM ^{r,s}_{(\omega )},\maclM ^{u,v}_{(\omega )})_{[\theta ]}
= \maclM ^{p,q}_{(\omega )}.
\end{equation*}
\end{enumerate}
Similar facts hold if the $\splM ^{p,q}_{(\omega )}$ spaces are replaced by
the $\splW ^{p,q}_{(\omega )}$ spaces.
\end{prop}

\par

The proof of Proposition \ref{p1.4}
can be found in \cite {Cordero,Feichtinger1,Feichtinger2,
Feichtinger3, Feichtinger4, Feichtinger5, Grochenig2, Toft2,
Toft4, Toft5, Toft8}.

\par

In fact, (1) follows from Gr{\"o}chenig's argument verbatim in
\cite[Proposition 11.3.2 (c)]{Grochenig2}. Note that the window
class $\splM ^1_{(v)}(\rr {2d})$ in (1) contains $\Sigma _1(\rr {2d})$,
which in turn contains $\maclS _{1/2}(\rr {2d})$. Furthermore, if in
addition $v\in \mascP (\rr {4d})$,
then $\splM ^1_{(v)}(\rr {2d})$ contains $\mascS (\rr {2d})$.

\par

The proof of (2) in \cite[Chapter 12]{Grochenig2} is based on Gabor
frames and formulated for polynomial type weights $\mascP
(\rr {4d})$. These arguments also hold for the broader weight class $\mascP
_E(\rr {4d})$. Another way to prove this is by means of
\cite[Lemma 11.3.3]{Grochenig2} and Young's inequality.

\par

The assertions (3)--(5) in Proposition \ref{p1.4} can be found for
more general weights in Theorem 4.17, and a combination of
Theorem 3.4 and Proposition 5.2  in \cite{Toft8}.

\par

\begin{rem}\label{remGSmodident}
Let $\mathcal P$ be the set of all
$\omega \in \mascP _E(\rr {4d})$ such that
$$
\omega (X,Y ) = e^{c(|X|^{1/s}+|Y|^{1/s})},
$$
for some $c>0$. (Note that this implies that $s\ge 1$.) Then
\begin{alignat*}{2}
\bigcap _{\omega \in \mathcal P}\splM ^{p,q}_{(\omega )}(\rr {2d}) &=
\Sigma _s(\rr {2d}),
&\quad \phantom{\text{and}}\quad
\bigcup _{\omega \in \mathcal P}\splM ^{p,q}_{(1/\omega )}(\rr {2d}) &=
\Sigma _s'(\rr {2d})
\\[1ex]
\bigcup _{\omega \in \mathcal P}\splM ^{p,q}_{(\omega )}(\rr {2d}) &=
\maclS _s(\rr {2d}),
&
\bigcap _{\omega \in \mathcal P}\splM ^{p,q}_{(1/\omega )}(\rr {2d}) &=
\maclS _s'(\rr {2d}),
\end{alignat*}
and for $\omega \in \mathcal P$
$$
\Sigma _s(\rr {2d})\subseteq \splM ^{p,q}_{(\omega )}(\rr {2d})
\subseteq
\maclS _s(\rr {2d}) \quad \text{and}\quad
\maclS _s'(\rr {2d}) \subseteq \splM ^{p,q}_{(1/\omega )}(\rr {2d}) \subseteq
\Sigma _s'(\rr {2d}).
$$
(Cf. \cite[Prop.~4.5]{CPRT10}, \cite[Prop.~4]{GZ}, \cite[Cor.~5.2]{Pilipovic1} and
\cite[Thm.~4.1]{Teo2}. See also \cite[Thm.~3.9]{Toft8} for an extension of these
inclusions to broader classes of Gelfand--Shilov and modulation spaces.)
\end{rem}

\par

We have the following result for the map $e^{it\scal {D_x}{D_\xi}}$ in
\eqref{calculitransform} when the domains are modulation spaces. We refer
to \cite[Proposition 1.7]{Toft5} for the proof (see also
\cite[Proposition 6.14]{Toft8}).

\par

\begin{prop}\label{propCalculiTransfMod}
Let $\omega _0\in \mascP _E(\rr
{4d})$, $p,q\in [1,\infty ]$, $t,t_1,t_2\in \mathbf R$, and set
$$
\omega _t(x,\xi ,\eta ,y)= \omega _0(x+ty,\xi +t\eta ,\eta ,y).
$$
The map $e^{it\scal {D_x}{D_\xi}}$ on $\maclS _{1/2}'(\rr {2d})$
restricts to a homeomorphism from $M^{p,q}_{(\omega
_0)}(\rr {2d})$ to $M^{p,q}_{(\omega _t)}(\rr {2d})$.

\par

In particular, if $a_1,a_2\in \maclS _{1/2}'(\rr {2d})$ satisfy
\eqref{calculitransform}, then $a_1\in
M^{p,q}_{(\omega _{t_1})}(\rr {2d})$, if and only if $a_2\in
M^{p,q}_{(\omega _{t_2})}(\rr {2d})$.
\end{prop}

\par

(Note that
in the equality of (2) in \cite[Proposition 6.14]{Toft8},
$y$ and $\eta$ should be interchanged in the last two arguments
in $\omega _0$.)

\medspace

By Proposition \ref{p1.4} (4) we have norm density of $\maclS _{1/2}$ in
$\splM ^{p,q}_{(\omega )}$ when $p,q<\infty$. We may relax the assumptions
on $p$, provided we replace the norm convergence with \emph{narrow}
convergence.
This concept, that allows us to approximate elements in $\splM ^{\infty,q}
_{(\omega )}(\rr {2d})$ for $1 \leq q < \infty$, 
is treated in \cite{Sjo1,Toft2,Toft4}, and, for
the current setup of possibly exponential weights, in \cite{Toft8}.
(Sj{\"o}strand's original definition in \cite{Sjo1} is somewhat different.)
Narrow convergence is defined by means of the function
$$
H_{a,\omega ,p}(Y) \equiv \| \maclV _\Phi a(\cdot,Y)\omega (\cdot,Y) \|
_{L^p(\rr {2d})}, \quad Y \in \rr {2d},
$$
for $a \in \maclS _{1/2}'(\rr {2d})$, $\omega \in \mascP _E(\rr {4d})$,
$\Phi \in \maclS _{1/2}(\rr {2d}) \setminus 0$ and $p\in [1,\infty]$.

\par

\begin{defn}\label{p2.1}
Let $p,q\in [1,\infty
]$, and $a,a_j\in
\splM ^{p,q}_{(\omega )}(\rr {2d})$, $j=1,2,\dots \ $. Then $a_j$
is said to \emph{converge narrowly} to $a$ with respect to $p,q$, $\Phi \in
\maclS _{1/2}(\rr {2d})\setminus 0$ and $\omega \in \mascP
_E(\rr {4d})$, if there exist $g_j,g \in L^q(\rr {2d})$ such that:

\begin{enumerate}
\item[{\rm{(1)}}] $a_j\to a$ in $\maclS _{1/2}'(\rr {2d})$ as $j\to \infty$;

\par

\item[{\rm{(2)}}] $H_{a_j,\omega ,p} \le g_j$ and $g_j \to g$ in $L^q(\rr
{2d})$ and a.{\,}e. as $j\to \infty$.
\end{enumerate}
\end{defn}

\par

\begin{prop}\label{narrowprop}
If $\omega \in \mascP _E(\rr {4d})$ and $1 \le q < \infty$ then the following
is true: 
\begin{enumerate}
\item[{\rm{(1)}}] $\maclS _{1/2}(\rr {2d})$ is
dense in $\splM ^{\infty,q}_{(\omega )}(\rr {2d})$ with respect to narrow
convergence;
\item[{\rm{(2)}}] $\splM ^{\infty,q}_{(\omega )}(\rr {2d})$ is sequentially
complete with respect to the topology defined by narrow convergence. 
\end{enumerate}
\end{prop}

\begin{proof}
Assertion (1) is a consequence of \cite[Definition 2.12 and Theorem 4.19]{Toft8}.

\par

To prove (2), let $\{a_n\}_{n=1}^\infty \subseteq  \maclS _{1/2}'(\rr {2d})$
be a Cauchy sequence with respect to narrow convergence. 
This means that 
\begin{equation}\label{cauchy1}
(a_n-a_k,\fy) \to 0, \quad n,k \to \infty, \quad \fy \in \maclS _{1/2}(\rr {2d}), 
\end{equation}
and there exists a sequence $\{g_n\} \subseteq L^q(\rr {2d})$ such that
$H_{a_n,\omega,\infty} \le g_n$,  and $\| g_n - g_k \|_{L^q} \to 0$ as
well as $g_n - g_k \to 0$ a.{\,}e.,  as $n,k \to \infty$.
By \eqref{cauchy1} and the completeness of $\maclS _{1/2}'(\rr {2d})$
there exists $a \in \maclS _{1/2}'(\rr {2d})$ such that $a_n \to a$ in
$\maclS _{1/2}'(\rr {2d})$ as $n \to \infty$, and by the completeness
of $L^q(\rr {2d})$ there exists $g \in L^q(\rr {2d})$ such that $g_n \to
g$ in $L^q(\rr {2d})$ and a.{\,}e. as $n \to \infty$. 
This shows that conditions (1) and (2) of Definition \ref{p2.1} are satisfied. 

\par

To show $a_n \to a$ narrowly as $n \to \infty$
it remains to prove $a \in \splM ^{\infty,q}_{(\omega )}(\rr {2d})$. 
We have for $Y \in \rr {2d}$
\begin{multline*}
H_{a,\omega,\infty} (Y) 
= \| \lim_{n \to \infty} \maclV _\Phi a_{n} (\cdot,Y)\omega (\cdot,Y)
\|_{L^\infty}
\\[1ex]
\le \limsup_{n \to \infty} \| \maclV _\Phi a_{n} (\cdot,Y)\omega (\cdot,Y)
\|_{L^\infty}
\\[1ex]
\le \| \limsup_{n \to \infty} |\maclV _\Phi a_{n} (\cdot,Y) | \omega (\cdot,Y)
\|_{L^\infty}
= H_{a,\omega,\infty} (Y). 
\end{multline*}
Since 
\begin{equation*}
H_{a,\omega,\infty} (Y) 
\le \liminf_{n \to \infty} H_{a_{n},\omega,\infty} (Y), \quad Y \in \rr {2d}, 
\end{equation*}
the limit $\lim_{n \to \infty} H_{a_{n},\omega,\infty} (Y)$ exists, 
so for almost all $Y \in \rr {2d}$ it follows that
\begin{equation*}
H_{a,\omega,\infty} (Y) =  \lim_{n \to \infty} H_{a_{n},\omega,\infty} (Y) 
\le \limsup_{n \to \infty} g_{n} (Y)
= g(Y). 
\end{equation*}
Since $g \in L^q(\rr {2d})$ we conclude that $H_{a,\omega,\infty} \in
L^q(\rr {2d})$ which means that $a \in \splM ^{\infty,q}_{(\omega )}(\rr {2d})$. 
\end{proof}

\par

%%%%%%%%%%%%%%%%%%%%%%%%%%
\section{Continuity for the Weyl product on
modulation spaces}\label{sec2}
%%%%%%%%%%%%%%%%%%%%%%%%%%

\par

In this section we deduce results on sufficient conditions for
continuity of the Weyl product on
modulation spaces, and the twisted convolution on Wiener amalgam spaces. 
The main results are Theorems
\ref{thm0.3}$'$ and \ref{thm0.4} concerning the Weyl product,
and Theorem \ref{thm0.3TwistConv} concerning the twisted
convolution.

\par

The first main result Theorem \ref{thm0.3}$'$ together with
Theorem \ref{thm0.4} is equivalent to Theorem \ref{thm0.3TwistConv}.
In the bilinear case, Theorem \ref{thm0.3}$'$ is the same as
Theorem \ref{thm0.3} in the introduction, and contains
\cite[Theorem 0.3$'$]{HTW} and Theorem \ref{thm0.4}.
On the other hand, in the multi-linear case with $N>2$, Theorems
\ref{thm0.3}$'$ and \ref{thm0.4} are distinct results with none of
them included in the other. 

\par

When proving Theorem \ref{thm0.3}$'$ we first need norm estimates. 
Then we prove the uniqueness of the extension, where generally norm
approximation not suffices, since the test function space may fail
to be dense in several of the domain spaces. The situation is
saved by a comprehensive argument based on narrow convergence.
First we prove the important
special cases Propositions \ref{prop1} and \ref{prop2} and then
we state and prove Theorem \ref{thm0.3}$'$.

\par

For $N \ge 2$ we let $\masfR _N$ be the H{\"o}lder--Young exponent
function
\begin{equation}\tag*{(\ref{HYfunctional})$'$}
\begin{aligned}
\masfR _N(\mathbf p) &= ({N-1})^{-1}\left ({\sum _{j=0}^N\frac
1{p_j}-1}\right ),
\\[1ex]
\mathbf p &= (p_0,p_1,\dots ,p_N)\in [1,\infty ]^{N+1}, 
\end{aligned}
\end{equation}
and we consider mappings of the form
\begin{equation}\tag*{(\ref{Weylmap})$'$}
(a_1,\dots ,a_N)\mapsto a_1\wpr \cdots \wpr a_N.
\end{equation}

We first show a formula for the STFT
of $a_1\wpr \cdots \wpr a_N$ expressed with
\begin{equation}\label{Fjdef}
F_j(X,Y) = \maclV_{\Phi _j}a_j (X+Y,X-Y).
\end{equation}

\par

\begin{lemma}\label{prodlemma}
Let $\Phi _j \in \maclS _{1/2}(\rr {2d})$, $j=1,\dots ,N$, $a_k \in
\maclS _{1/2}'(\rr {2d})$ for some $1 \le k \le N$, and $a_j \in
\maclS _{1/2}(\rr {2d})$ for $j \in \{1,\dots ,N\} \setminus k$. 
Suppose
$$
\Phi _0 = \pi ^{(N-1)d}\Phi _1\wpr \cdots \wpr \Phi _N\quad
\text{and}\quad
a_0 = a_1\wpr \cdots \wpr a_N.
$$
If $F_j$ are given by \eqref{Fjdef} then
\begin{multline}\label{STFTintegral}
F_0(X_N,X_0)
\\[1ex]
=\idotsint _{\rr {2(N-1)d}}e^{2 i Q(X_0,\dots  ,X_N)}
\prod _{j=1}^NF_j(X_j,X_{j-1}) \, dX_1
\cdots dX_{N-1}
\end{multline}
with 
$$
Q(X_0,\dots,X_N)=\sum_{j=1}^{N-1}\sigma(X_j-X_0,X_{j+1}-X_0).
$$ 
\end{lemma}

\par

\begin{proof}
The result follows in the case $N=2$ by letting $X=X_2+X_0$ and
$Y=X_2-X_0$ in \cite[Lemma 2.1]{HTW}. For $N > 2$ the
result follows from straight-forward computations and induction.
\end{proof}

\par

Next we use the previous lemma to find sufficient conditions
for the extension of \eqref{Weylmap}$'$ to modulation spaces.
The integral representation of $V_{\Phi _0}a_0$ in the previous
lemma leads to the weight condition
\begin{multline}\tag*{(\ref{weightcond})$'$}
1 \lesssim \omega _0(X_N+X_0,X_N-X_0)\prod _{j=1}^N
\omega _j(X_j+X_{j-1},X_j-X_{j-1}),
\\[1ex]
X_0,X_1,\dots ,X_N\in \rr {2d}.
\end{multline}

\par

The following result is a generalization of \cite[Proposition 0.1]{HTW}.

\par

\begin{prop}\label{prop1}
Let $p_j,q_j\in [1,\infty ]$, $j=0,1,\dots , N$, and suppose
$$
\masfR _N(\mathbf q')\le 0\le \masfR _N(\mathbf p).
$$
Let $\omega _j$, $j=0,1,\dots ,N$, and suppose
\eqref{weightcond}$'$ holds. Then the map \eqref{Weylmap}$'$
from $\maclS _{1/2}(\rr {2d}) \times \cdots \times
\maclS _{1/2}(\rr {2d})$ to $\maclS _{1/2}(\rr {2d})$
extends uniquely to a continuous and
associative map from $\splM ^{p_1,q_1}_{(\omega _1)}(\rr {2d})
\times \cdots \times
\splM ^{p_N,q_N}_{(\omega _N)}(\rr {2d})$ to $\splM ^{p_0',q_0'}
_{(1/\omega _0)}(\rr {2d})$.
\end{prop}

\par

The associativity means that for any
product \eqref{Weylmap}$'$, where the factors $a_j$ satisfy the hypotheses, the subproduct
$$
a_{k_1}\wpr a_{k_1+1} \wpr  \cdots \wpr a_{k_2}
$$
is well defined as a distribution for any $1\le k_1 \le k_2\le N$, and
$$
a_1\wpr \cdots \wpr a_N = (a_1\wpr \cdots \wpr a_k)
\wpr (a_{k+1}\wpr \cdots \wpr a_N),
$$
for any $1\le k\le N-1$.
 
\par

To prove the uniqueness claim we need the following two lemmas, 
the first of which is a generalization of Lebesgue's dominated
convergence theorem.

\par

\begin{lemma}\label{GenLebThm}
Let $0< q < \infty$, let $\{ f_n \} _{n\ge 0}$ and $\{ g_n \} _{n\ge 0}$
be sequences in $L^q ( \rr d )$ such that
\begin{equation*}
\begin{aligned}
& \lim _{n\to \infty} \nm {g_n-g}{L^q(\rr d)} = 0,
\quad
\lim _{n\to \infty} g_n = g \quad a.{\,}e.,
\\[1ex]
& |f_n| \le g_n,
\quad \text{and}\quad
\lim _{n\to \infty}f_n =f \quad a.{\,}e.,
\end{aligned}
\end{equation*}
for some measurable functions $f$ and $g$. Then $f\in L^q (\rr d)$
and $$
\lim _{n\to \infty} \nm {f_{n}-f}{L^q (\rr d)} = 0.
$$
\end{lemma}

\par

\begin{proof}
The result follows from an argument based on Fatou's lemma applied on
\begin{equation*}
\int 2^q g(x)^q \, dx 
= \int \liminf_{n \to \infty} \left( (g_{n}(x) + g(x) )^q -
| f_{n}(x)-f(x) |^q \right )\, dx. 
\end{equation*}
\end{proof}

\par

\begin{lemma}\label{intconv}
Let $1 < q \leq \infty$, let $f\in L^{q'}(\rr d)$,
and let $\{ g_n \} _{n\ge 0}$ be a sequence in $L^{q}(\rr d)$ 
such that
\begin{equation*}
\sup _{n} \nm {g_n}{L^q(\rr d)} < \infty
\quad \mbox{and} \quad
\lim _{n\to \infty} g_n = g \quad a.{\,}e. ,
\end{equation*}
for some measurable function $g$. Then $g\in L^{q} (\rr d)$
and 
$$
\lim _{n \to \infty} \int_{\rr d} (g_{n}(x)-g(x)) \, f(x) \, dx = 0.
$$
\end{lemma}

\par

\begin{proof}
The result follows from a combination of
Egorov's theorem and the facts that for any
$\ep >0$ there is a ball $B\subseteq \rr d$ such that
$$
\nm f{L^{q'}(\rr d\setminus B)}<\ep ,
$$
and
$$
\lim _{|E|\to 0}\nm f{L^{q'}(E)} = 0,
$$
where $|E|$ denotes the volume of the measurable set
$E\subseteq \rr d$. The details are left for the reader. 
\end{proof}

\par

\begin{proof}[Proof of Proposition \ref{prop1}]
By Proposition \ref{p1.4} (2) we may assume that $\masfR _N(\mathbf p)
=\masfR _N(\mathbf q')=0$, which will allow us to use H\"older's and
Young's inequalities.

\par

Let
$a_1,\dots ,a_N\in \maclS _{1/2}(\rr {2d})$. By replacing $X_j$ with
$X_j+X_0$ in \eqref{weightcond}$'$, $j=1,\dots ,N$, and then replacing
$2X_0$ with $X_0$, we get
\begin{multline}\label{weightcond1A}
1 \lesssim \omega _0(X_N+X_0,X_N)\omega _1(X_1+X_0,X_1)
\prod _{j=2}^N \omega _j(X_j+X_{j-1}+X_0,X_j-X_{j-1}),
\\[1ex]
X_0,\dots ,X_N\in \rr {2d}.
\end{multline}
Let $\Phi _j$, $j=0,\dots,N$, be as in Lemma \ref{prodlemma}. Set
\begin{align*}
G_j(X,Y) &\equiv | \maclV_{\Phi _j}a_j(X,Y)|\omega _j(X,Y),\qquad
g_j(Y)\equiv \nm {G_j(\cdo ,Y)}{L^{p_j}} ,
\intertext{for $j=1,\dots ,N$, and}
K(X_0,\dots ,X_N) &= G_1(X_1+X_0,X_1)
\prod _{j=2}^NG_j(X_j+X_{j-1}+X_0,X_j-X_{j-1}).
\end{align*}
Then Lemma \ref{prodlemma} gives
\begin{multline*}
|\maclV_{\Phi _0}a_0(X_N+X_0,X_N)|/\omega _0(X_N+X_0,X_N)
\\[1ex]
\lesssim \idotsint _{\rr {2(N-1)d}} K(X_0,\dots ,X_N)\, dX_1
\cdots dX_{N-1}.
\end{multline*}

\par

Taking the $L^{p_0'}$ norm in the first variable gives, using Minkowski's
and H{\"o}lder's inequalities,
\begin{multline*}
\nm {\maclV_{\Phi _0}a_0(\cdo ,X_N)|/\omega _0(\cdo ,X_N)}{L^{p_0'}}
\\[1ex]
\lesssim \idotsint _{\rr {2(N-1)d}} \nm {K(\cdo ,X_1,\dots ,X_N)}
{L^{p_0'}}\, dX_1 \cdots dX_{N-1}
\\[1ex]
\le \idotsint _{\rr {2(N-1)d}} g_1(X_1)
\prod _{j=2}^N g_j(X_j-X_{j-1}) \, dX_1
\cdots dX_{N-1}
\\[1ex]
= (g_1*\cdots *g_N)(X_N).
\end{multline*}
Applying the $L^{q_0'}$ norm and using Young's inequality we get
\begin{multline}\label{prodest}
\nm {a_0}{\splM ^{p_0',q_0'}_{(1/\omega _0)}} \lesssim
\nm {g_1}{L^{q_1}}\cdots \nm {g_N}{L^{q_N}} =
\nm {a_1}{\splM ^{p_1,q_1}_{(\omega _1)}}
\cdots
\nm {a_N}{\splM ^{p_N,q_N}_{(\omega _N)}}.
\end{multline}

\par

The result now follows in the case when $p_j,q_j<\infty$ for $j=1,\dots ,N$,
from the estimate \eqref{prodest} and the fact that
$\maclS _{1/2}$ is dense in $\splM ^{p_j,q_j}_{(\omega _j)}$.
In the case when at least one $p_j$ or $q_j$ attain
$\infty$ for some $j=1,\dots ,N$, \eqref{prodest} still holds
when $a_j\in \maclS _{1/2}$, $j=1,\dots ,N$, and
the Hahn--Banach theorem and duality guarantee the existence of a
continuous extension.

\par

We must prove its uniqueness and associativity. First we observe
that the assumption $\masfR _N(\mathbf q')=0$ is equivalent
to $\sum_{j=0}^N 1/q_j = N$, so $q_k = \infty$ may hold for at most one
$k$, and in that case $q_j=1$ must hold for $j \in \{0,\dots,N \} \setminus
k$. If $q_0>1$ then $q_j < \infty$ for $1 \le j \le N$. 
So either the uniqueness concerns the inclusion
\begin{equation}\label{allfinite}
\splM ^{p_1,q_1}_{(\omega _1)} \wpr  \cdots \wpr \splM ^{p_N,q_N}
_{(\omega _N)} \subseteq \splM ^{p_0',q_0'}_{(1/\omega _0)}, \quad
q_j < \infty, \quad 1 \le j \le N, \quad q_0>1, 
\end{equation}
or 
\begin{equation}\label{oneinfinite}
\splM ^{p_1,q_1}_{(\omega _1)} \wpr  \cdots \wpr \splM ^{p_N,q_N}
_{(\omega _N)} \subseteq \splM ^{p_0',\infty}_{(1/\omega _0)},
\quad q_j < \infty, \quad 1 \le j \le N, 
\end{equation}
or, for a unique $k$ such that $1 \le k \le N$, 
\begin{equation}\label{twoinfinite}
\splM ^{p_1,1}_{(\omega _1)} \wpr  \cdots \wpr \splM ^{p_k,\infty}
_{(\omega _k)} \wpr \cdots \wpr \splM ^{p_N,1}_{(\omega _N)}
\subseteq \splM ^{p_0',\infty}_{(1/\omega _0)}.  
\end{equation}

\par

First we consider \eqref{allfinite}. 
For all $j$ such that $p_j < \infty$ we may extend the Weyl product
uniquely from $a_j \in \maclS _{1/2}$ to $a_j \in \splM ^{p_j,q_j}
_{(\omega _j)}$ as in the first part of the proof, and for the remaining
$j$ we extend the Weyl product from $a_j \in \maclS _{1/2}$ to $a_j
\in \splM ^{\infty,q_j}_{(\omega _j)}$ using narrow convergence, as
follows. By induction it suffices to perform the extension for some
$j \in \{ 1,\dots,N\}$ from $a_j \in \maclS _{1/2}$  to 
$a_j \in \splM ^{\infty,q_j}_{(\omega _j)}$.

\par

Assume for simplicity that $j=1$. 
We may assume $q_1>1$. In fact, our aim is to prove uniqueness
only,  so if $q_1=1$ we may by Proposition \ref{p1.4} (2) consider the
first factor $a_1$ as an element in $\splM ^{\infty,\tilde q_1}
_{(\omega _1)}$ with the exponent $q_1$ modified as $1/\tilde
q_1=1/q_1- \ep < 1$, where $\ep>0$ is so small that we still have
$1/\tilde q_0=1/q_0+ \ep <1$ for the modification $\tilde q_0$ of
the exponent $q_0$.  

\par

Take a sequence
$\{ a_{1,n} \} _{n=1}^\infty$ that converges narrowly to $a_1 \in \splM
^{\infty,q_1}_{(\omega _1)}$. 
By Definition \ref{p2.1} this means that
$a_{1,n} \to a_1$ in $\maclS _{1/2}'(\rr {2d})$ as $n \to \infty$, 
and the existence of $g_{1,n}, g_1 \in L^{q_1}(\rr {2d})$ such that 
\begin{equation*}
\nm { \maclV_{\Phi _1}a_{1,n}(\cdo,Y) \omega _1(\cdo,Y) }{L^{\infty}}
\le g_{1,n}(Y)
\end{equation*}
and $g_{1,n} \to g_1$ in $L^{q_1}(\rr {2d})$ as well as a.{\,}e. as
$n \to \infty$. Set 
\begin{equation*}
g_j (Y) = \nm { \maclV_{\Phi _j}a_j(\cdo,X) \omega _j(\cdo,Y) }
{L^{\infty}}, \quad j=2, \dots,N, 
\end{equation*}
and define $a_{0,n} = a_{1,n} \wpr a_2\wpr \cdots \wpr a_N$.  
Lemma \ref{STFTintegral} and the definitions above yield 
\begin{align*}
& \nm{ \maclV_{\Phi _0} a_{0,n}(\cdo,X_N) / \omega _0(\cdo,X_N) }
{L^\infty}
\\[1ex]
& \lesssim \idotsint _{\rr {2(N-1)d}}   
g_{1,n} (X_1)
\prod _{j=2}^N g_j(X_j-X_{j-1})
\, dX_1 \cdots dX_{N-1} 
\\[1ex]
& = g_{1,n} * g_2 * \cdots * g_N (X_N). 
\end{align*}
From $g_{1,n} \to g_1$ in $L^{q_1}$ as $n \to \infty$ and Young's
inequality, we may conclude that $g_{1,n} * g_2 * \cdots * g_N \to
g_1 * g_2 * \cdots * g_N$ in $L^{q_0'}(\rr {2d})$. 

\par

The assumption $\sum_{j=0}^N 1/q_j = N$ implies $1/q_j \geq 1/q_0'$
for any $1 \leq j \leq N$. 
Due to Proposition \ref{p1.4} (2) we may therefore assume that
$g_2 \in L^{q}$ with $1/q=1/q_2 -1/q_0'$. Then Young's inequality
guarantees that $g_2 * \cdots * g_N \in L^{q_1'} (\rr {2d})$.
It now follows from Lemma \ref{intconv} that $g_{1,n} * g_2 * \cdots
* g_N \to g_1 * g_2 * \cdots * g_N$ a.{\,}e. So we have shown that
the sequence $\{ a_{0,n}\}$ satisfies condition (2) in Definition \ref{p2.1}
for the modulation space $\splM ^{\infty,q_0'}_{(1/\omega _0)}
(\rr {2d})$. 

\par

Let $\fy \in \maclS _{1/2}(\rr {2d})$. Our plan is to show that
$(a_{0,n} - a_{0,k}, \fy) \to 0$ as $n,k \to \infty$. 
Together with the conclusions above this will imply 
that $\{ a_{0,n} \}$ is a Cauchy sequence with
respect to narrow convergence. 
Proposition \ref{narrowprop} (2) then
guarantees that it has a narrow limit $a_0 \in \splM ^{\infty,q_0'}
_{(1/\omega _0)}(\rr {2d})$, which we use as the definition of $a_1
\wpr \cdots \wpr a_N$. It follows that the Weyl product extends uniquely 
from $a_1 \in \maclS _{1/2}$ to $a_1 \in \splM ^{\infty,q_1}
_{(\omega _1)}$. 

\par

By Lemma \ref{prodlemma} we have
\begin{multline}\label{cjmnb}
(a_{0,n} - a_{0,k},\fy) = C_{\Phi _0}(\maclV _{\Phi _0} (a_{0,n}
- a_{0,k}),\maclV _{\Phi _0} \fy)
\\[1ex]
= C_{\Phi _0}\idotsint _{\rr {2(N+1)d}} e^{2 i Q(X_0,\dots  ,X_N)}
H_{n,k}(X_0,\dots , X_N)\, dX_0\cdots dX_N,
\end{multline}
where 
\begin{multline*}
H_{n,k}(X_0,\dots , X_N)
= 4^d
\maclV _{\Phi _1} (a_{1,n}- a_{1,k})(X_1+X_{0},X_1-X_{0})\times
\\[1ex]
\times \left ( \prod _{j=2}^N
F_j(X_j,X_{j-1}) \right ) 
\overline {\maclV _{\Phi _0} \fy (X_N+X_0,X_N-X_0)}.
\end{multline*}

\par

By the narrow convergence we have $\maclV _{\Phi _1}a_{1,n}\to
\maclV _{\Phi _1}a_{1}$ pointwise as $n \to \infty$, which implies
that
\begin{equation}\label{HjmnLimit}
\lim _{n,k\to \infty }H_{n,k}(X_0,\dots , X_N)= 0, \quad
(X_0,\dots , X_N) \in \rr {2(N+1)d}. 
\end{equation}
If we define
$$
G(X,Y) =  |\maclV _{\Phi _0} \fy(X+Y,X-Y)| \, \omega _0(X+Y,X-Y),
$$
then $|H_{n,k}| \lesssim K_{n,k}$, where
\begin{multline*}
K_{n,k}(X_0,\dots , X_N)
\\[1ex]
\equiv 
(g_{1,n}(X_1-X_0) + g_{1,k}(X_1-X_0) )
\left (\prod _{j=2}^N g_{j}(X_j-X_{j-1})
\right ) |G(X_N,X_0)|.
\end{multline*}
By Young's inequality and the assumption $g_{1,n} \to g_1$ in
$L^{q_1}$, $K_{n,k}$ has a limit
in $L^1(\rr {2(N+1)d})$, denoted $K$. 
By the assumption $g_{1,n} \to g_1$ a.{\,}e., $K_{n,k} \to K$
a.{\,}e. as $n,k \to \infty$. Hence \eqref{cjmnb}, \eqref{HjmnLimit}
and Lemma \ref{GenLebThm}
imply that $(a_{0,n} - a_{0,k}, \fy) \to 0$ as $n,k \to \infty$. 

\par

By the same arguments it follows that the integral formula
\eqref{STFTintegral} holds for the extension for almost all $(X_N,X_0)
\in \rr {4d}$.  This finishes
the proof of the uniqueness of the extended Weyl product inclusion
\eqref{allfinite}.

\par

The uniqueness in the cases \eqref{oneinfinite} and
\eqref{twoinfinite} follow from the uniqueness in the case
\eqref{allfinite} and duality. 

\medspace

It remains to prove the asserted associativity, and first we need to
prove that any subproduct of $a_1\wpr
\cdots \wpr a_N$ is well defined. We observe that
\eqref{weightcond}$'$ can be written as
$$
1 \lesssim \omega _0(X_N+X_0,X_N-X_0) \, \vartheta _1(X_0,\dots ,X_k) \, \vartheta _2(X_k,\dots ,X_N)
$$
for
\begin{align*}
\vartheta _1(X_0,\dots ,X_k) &= \prod _{j=1}^k\omega _j
(X_j+X_{j-1},X_j-X_{j-1})
\intertext{and}
\vartheta _2(X_k,\dots ,X_N) &= \prod _{j=k+1}^N\omega _j
(X_j+X_{j-1},X_j-X_{j-1}),
\end{align*}
and any $1 \le k \le N-1$. If $\vartheta$ is defined by
$$
\vartheta (X_k+X_0,X_k-X_0)\equiv \inf \vartheta _1(X_0,\dots ,X_k)
$$
where the infimum is taken over all $X_1,\dots ,X_{k-1}\in \rr {2d}$, it follows
from \eqref{weightcond}$'$ that
\begin{align*}
& 1 \lesssim \vartheta (X_k+X_0,X_k-X_0)^{-1} \prod _{j=1}^k\omega _j
(X_j+X_{j-1},X_j-X_{j-1}),
\\[1ex]
\phantom{kk} & \text{$\ $} \qquad \qquad \qquad \qquad
\qquad \qquad \qquad \qquad \qquad \qquad
X_0, X_1, \dots,X_k \in \rr {2d}, 
\intertext{and}
& 1\lesssim \omega _0(X_N+X_0,X_N-X_0) \, \vartheta
(X_k+X_0,X_k-X_0) \prod _{j=k+1}^N\omega _j
(X_j+X_{j-1},X_j-X_{j-1}),
\\[1ex]
\phantom{kk} & \text{$\ $} \qquad \qquad \qquad \qquad
\qquad \qquad \qquad \qquad \quad
X_0, X_k, ,X_{k+1}, \dots,X_N \in \rr {2d}. 
\end{align*}
Note that  $\vartheta \in \mascP _E(\rr {4d})$ by the assumptions.

\par

It now follows from the first part of the proof that
\begin{align*}
a_1\wpr \cdots \wpr a_k &\in \splM ^{r,s}_{(\vartheta )}
\intertext{and}
b\wpr a_{k+1}\wpr \cdots \wpr a_N &\in \splM ^{p_0',q_0'}_{(1/\omega _0 )},
\end{align*}
when $a_j\in \splM ^{p_j.q_j}_{(\omega _j)}$, $b\in \splM ^{r,s}_{(\vartheta )}$,
and $r,s \in [1,\infty]$ are defined by 
$$
\frac 1r = \sum _{j=1}^k\frac 1{p_j}
\quad \text{and}\quad
\frac 1{s'} = \sum _{j=1}^k\frac 1{q_j'}.
$$
This shows that $a_1\wpr \cdots \wpr a_k$ and
$a_{k+1}\wpr \cdots \wpr a_N$ are well-defined as elements in
appropriate modulation spaces.

\par

The asserted associativity now follows from the density arguments in the
proof of the uniqueness, 
and the fact that the Weyl product
is associative on $\maclS _{1/2}$. 
\end{proof}

\par

For appropriate weights $\omega$ the space
$\splM ^2_{(\omega )}(\rr {2d})$ consists of symbols of Hilbert--Schmidt
operators acting between certain modulation spaces (cf. \cite{Toft5,Toft9}). 
The following proposition, with $p_j=q_j=2$ for $j=0,\dots,N$, is a
manifestation of the fact that Hilbert--Schmidt
operators are closed under composition.
The result in that special case is a consequence of
\cite[Proposition 0.2]{HTW}, which concerns $N=2$, with $p_j=q_j=2$,
$j=0,1,2$, and induction. The general result relaxes the
assumption on the exponents, and is an essential step towards
the improvement Theorem \ref{thm0.3}$'$ below.

\par

\begin{prop}\label{prop2}
Let $p_j,q_j\in [1,\infty ]$, $j=0,1,\dots , N$, and suppose
\begin{equation}\label{pqconditionsA}
\max \left ( \masfR _N(\mathbf q') ,0 \right )
\le  \min _{j=0,1,\dots ,N} \left ( \frac 1{p_j},\frac 1{p_j'},\frac 1{q_j},
\frac 1{q_j'}, \masfR _N(\mathbf p)\right ).
\end{equation}
Let $\omega _j \in \mascP _E(\rr {4d})$, $j=0,1,\dots ,N$, and
suppose \eqref{weightcond}$'$ holds. 
Then the map \eqref{Weylmap}$'$ 
from $\maclS _{1/2}(\rr {2d}) \times \cdots \times
\maclS _{1/2}(\rr {2d})$ to $\maclS _{1/2}(\rr {2d})$
extends uniquely to a continuous and associative map from
$\splM ^{p_1,q_1} _{(\omega _1)}(\rr {2d}) \times \cdots
\times \splM ^{p_N,q_N} _{(\omega _N)}(\rr {2d})$ to
$\splM ^{p_0',q_0'} _{(1/\omega _0)}(\rr {2d})$.
\end{prop}

\par

\begin{proof}
First we prove the result for $p_j=q_j=2$ for all $0 \le j \le N$.
Let $a_j\in \maclS _{1/2}$, $j=1,\dots,N$, and let
$$
G_j(X,Y)=|F_j(X,Y)\omega _j(X+Y,X-Y)|,\quad j=1,\dots ,N,
$$
where $F_j$ are given by \eqref{Fjdef}. 
Lemma \ref{prodlemma} and repeated application of H\"older's
inequality give
\begin{multline*}
| F_0(X_N,X_0)|/\omega _0(X_N+X_0,X_N-X_0)
\\[1ex]
\lesssim
\idotsint _{\rr {2(N-1)d}}
\left (\prod _{j=1}^N G_j(X_j,X_{j-1})\right )\, dX_1
\cdots dX_{N-1}
\\[1ex]
\le \nm {G_1(\cdo ,X_0)}{L^2(\rr {2d})} \nm {G_N(X_N,\cdo )}{L^2(\rr {2d})}
\prod _{j=2}^{N-1} \nm {G_j}{L^2(\rr {4d})}.
\end{multline*}
Taking the $L^2(\rr {4d})$ norm gives
$$
\nm {a_0}{\splM ^2_{(1/\omega _0)}}
\lesssim \prod _{j=1}^N\nm {G_j}{L^2}\asymp
\prod _{j=1}^N\nm {a_j}{\splM ^2_{(\omega _j)}}.
$$
The claim follows from this estimate and the fact that $\maclS _{1/2}$
is dense in $\splM ^2_{(\omega _j)}$.

\par

The proof of the general case is based on multi-linear interpolation between
the case $p_j=q_j=2$ for $0 \le j \le N$ and Proposition \ref{prop1}.

\par

More precisely, by Proposition \ref{prop1}
and the first part of this proof we have
$$
\splM ^{r_1,s_1}_{(\omega _1)} \wpr \cdots \wpr \splM ^{r_N,s_N}_{(\omega _N)}
\subseteq \splM ^{r_0',s_0'} _{(1/\omega _0)}
\quad \text{and}\quad
\splM ^{2,2}_{(\omega _1)} \wpr \cdots \wpr \splM ^{2,2}_{(\omega _N)}
\subseteq \splM ^{2,2} _{(1/\omega _0)},
$$
when $r_j,s_j \in [1,\infty]$, $j=0,1,\dots,N$, and
\begin{equation}\label{rsuvconds}
\masfR _N(\mathbf s')\le 0\le \masfR _N(\mathbf r).
\end{equation}

\par

By multi-linear interpolation, using \cite[Theorem 4.4.1]{BeLo} and
Proposition \ref{p1.4} (5), we get
\begin{equation}\label{Mcont}
\splM ^{p_1,q_1}_{(\omega _1)} \wpr \cdots \wpr \splM ^{p_N,q_N}_{(\omega _N)}
\subseteq \splM ^{p_0',q_0'} _{(1/\omega _0)}
\end{equation}
when
\begin{equation}\label{intpqcond}
\frac 1{p_j} = \frac {1-\theta}{r_j}+\frac {\theta}{2}\quad \text{and}\quad
\frac 1{q_j} = \frac {1-\theta}{s_j}+\frac {\theta}{2},\quad 0\le \theta \le 1,
\end{equation}
$j=0,1,\dots,N$. 

Suppose $p_j,q_j\in [1,\infty]$, $j=0,1,\dots,N$, satisfy \eqref{pqconditionsA}.
We have to show that there
exist $0 \leq \theta \leq 1$, $r_j\in [1,\infty ]$ and $s_j\in [1,\infty ]$,
$j=0,1,\dots,N$, such that \eqref{rsuvconds} and \eqref{intpqcond}
are satisfied, after which \eqref{Mcont} follows by multi-linear
interpolation. Our plan is to first find an appropriate $\theta$, and then
find $\mathbf r$ and $\mathbf s$ with the right properties.

\par

We have $r_j\in [1,\infty ]$ if and only if
$$
0 \leq \frac{1-\theta}{r_j} = \frac1{p_j} - \frac{\theta}{2} \leq 1-\theta,
$$
i.{\,}e. $\theta/2 \leq \min(1/p_j, 1/p_j')$, and likewise $s_j\in [1,\infty ]$ if
and only if $\theta/2 \leq \min(1/q_j, 1/q_j')$.
Since $\masfR_N (\mathbf q') \leq 1/2$ as a consequence of
\eqref{pqconditionsA}, there exists $\theta \in [0,1]$ such that
\begin{equation*}
\masfR _N (\mathbf q') \leq \frac{\theta}{2}
\leq \min _{j=0,1,\dots ,N} \left ( \frac {1}{q_j},\frac {1}{q_j'},\frac {1}{p_j},
\frac {1}{p_j'}, \masfR _N(\mathbf p)\right )
\end{equation*}
again by the assumption \eqref{pqconditionsA}. With such a choice of
$\theta$ we have $r_j, s_j \in [1,\infty ]$ for $j=0,1,\dots ,N$, and
\begin{equation*}
\masfR_N (\mathbf q') \le \theta /2\le \masfR_N (\mathbf p).
\end{equation*}
This gives
\begin{equation*}
\masfR _N(\mathbf r) =\frac 1{1-\theta}\left (\masfR _N(\mathbf p)-
\frac {\theta}2\right )\ge 0
\end{equation*}
and
\begin{equation*}
\masfR _N(\mathbf s') =\frac 1{1-\theta}\left (\masfR _N(\mathbf q')-
\frac {\theta}2\right )
\le \frac 1{1-\theta}\left (\frac \theta 2 -
\frac {\theta }2\right ) = 0.
\end{equation*}
Hence \eqref{rsuvconds} and \eqref{intpqcond} are satisfied, and \eqref{Mcont}
follows. Thus \eqref{pqconditionsA} implies \eqref{Mcont}.

\par

It remains to prove the associativity. 
If $\masfR _N(\mathbf q ')\le 0\le \masfR _N(\mathbf p )$ the associativity follows
from Proposition \ref{prop1}, and if
$p_j=q_j=2$, $j=0,\dots ,N$, the associativity follows from
the associativity of the Weyl product on $\maclS _{1/2}$, and the fact
that $\maclS _{1/2}$ is dense in $\splM ^{2,2}_{(\omega _j)}$ for every
$j$.
The associativity now follows in general from the fact that the general
case is an interpolation between the latter two cases. 
\end{proof}

\par

\begin{rem}
A crucial step in the proof is the fact that
\eqref{pqconditionsA} implies that $\theta$ and $\mathbf r,\mathbf
s\in [1,\infty]^{N+1}$ can be chosen such that \eqref{rsuvconds}
and \eqref{intpqcond} holds.
On the other hand, by straight-forward computations it follows that if
\eqref{rsuvconds} and \eqref{intpqcond} are fulfilled, then
\eqref{pqconditionsA} holds. 
\end{rem}

\par

\begin{rem}\label{remextension}
We note that Proposition \ref{prop2} extends \cite[Theorem 0.3$'$]{HTW}. 
The latter result asserts that if $N=2$,
\begin{equation}\label{thm0.3HTWformulas}
\begin{aligned}
\masfR _2(\mathbf p) &= \masfR _2(\mathbf q'), \qquad q_1,q_2\le q_0', 
\\[1ex]
0 \le \masfR _2(\mathbf p)&\le \frac 1{p_j},\frac 1{q_j},\frac 1{p_0'},
\frac 1{q_0'}\le 1-\masfR _2(\mathbf q '),\quad j=1,2 ,
\end{aligned}
\end{equation}
and $\omega _j$, $j=0,1,2$, are weights that satisfy
\eqref{weightcond}, then the map
\eqref{Weylmap} extends to a continuous map from
$\splM ^{p_1,q_1}_{(\omega _1)}\times \splM ^{p_1,q_1}
_{(\omega _1)}$ to $\splM ^{p_0',q_0'}_{(1/\omega _0)}$.
We claim that \eqref{thm0.3HTWformulas}
implies \eqref{pqconditionsA} when $N=2$, which means
that Proposition \ref{prop2} extends \cite[Theorem 0.3$'$]{HTW}.

\par

In fact, by the last inequality in \eqref{thm0.3HTWformulas} we get
$$
\masfR _2(\mathbf q')\le \frac 1{p_j'},\frac 1{q_j'},
\frac 1{p_0},\frac 1{q_0}, \quad j=1,2.
$$
A combination of these inequalities gives
$$
\masfR_2 (\mathbf q') = \masfR_2 (\mathbf p) \le \min _{j=0,1,2}\left (
\frac 1{p_j},\frac 1{p_j'},\frac 1{q_j}\frac 1{q_j'}\right ),
$$
and it follows that the hypothesis  \eqref{pqconditionsA} in Proposition
\ref{prop2} is fulfilled for $N=2$. 
\end{rem}

\par

Next we prove that the conclusion of Proposition \ref{prop2} holds
under assumptions that are weaker than \eqref{pqconditionsA}. 
The following lemma shows that we may omit the
condition $\masfR _N(\mathbf q')\le \min_{0 \le j \le N} (1/q_j)$ in
\eqref{pqconditionsA}.

\par

\begin{lemma}\label{lemmaAthm0.3}
Let $N \ge 2$, $x_j\in [0,1]$, $j=0,\dots ,N$ and consider the inequalities:
\begin{enumerate}
\item[{\rm{(1)}}] $\displaystyle{(N-1)^{-1}\left (\sum _{k=0}^Nx_k -1\right )\le 
\min _{0\le j\le N}x_j}$;

\vrum

\item[{\rm{(2)}}] $x_j+x_k\le 1$, for all $k\neq j$;

\vrum

\item[{\rm{(3)}}] $\displaystyle{(N-1)^{-1}\left (\sum _{k=0}^Nx_k -1\right )
\le \min _{0\le j\le N}(1-x_j)}$.
\end{enumerate}
Then
$$
{\rm{(1)}}\Longrightarrow {\rm{(2)}} \Longrightarrow {\rm{(3)}}.
$$

\par

If $N=2$ then ${\rm{(1)}}$ and ${\rm{(2)}}$ are
equivalent.
\end{lemma}

\par

\begin{proof}
Assume that (1) holds but (2) fails. Then $x_j+x_k>1$ for some $j\neq k$.
By renumbering we may assume that $x_2\le x_j$ for every $j$, and
that $x_0+x_1>1$. Then (1) gives
\begin{equation*}
(N-1)x_2\le \sum _{k=2}^Nx_k < \sum _{k=0}^Nx_k -1
\le (N-1)x_2
\end{equation*}
which is a contradiction. Hence the assumption
$x_0+x_1>1$ must be wrong and it follows that (1) implies (2).

\par

Now assume that (2) holds. Then
$$
x_j\le 1-x_k,\quad k\neq j.
$$
This gives
\begin{equation*}
\sum _{j\neq k}x_j\le N(1-x_k)\quad
\Longleftrightarrow \quad
\sum _{j=0}^N x_j-1\le (N-1)(1-x_k)
\end{equation*}
for any $k$, so (3) holds.

\par

Finally, if $j\neq k$, $N=2$ and (2) holds, then $x_j+x_k\le 1$,
which gives $x_j+x_k+x_l-1\le x_l$, $l=0,1,2$. In particular,
$$
\sum _{k=0}^2x_k -1\le x_l,\qquad l=0,1,2,
$$
and (1) follows.
\end{proof}

\par

The next result is one of two principal theorems on sufficient conditions for
continuity. It shows that one can eliminate some conditions on the Lebesgue
exponents in Proposition \ref{prop2}. In particular the result extends
\cite[Theorem 0.3$'$]{HTW} in view of Remark \ref{remextension}.

\par

\renewcommand{\rubrik}{Theorem \ref{thm0.3}$'$}

\begin{tom}
Let $p_j,q_j\in [1,\infty ]$, $j=0,1,\dots , N$, and suppose
\begin{equation}\tag*{(\ref{pqconditions})$'$}
\max \left ( \masfR _N(\mathbf q') ,0 \right )
\le  \min _{j=0,1,\dots ,N} \left ( \frac 1{p_j},\frac 1{q_j'},
\masfR _N(\mathbf p)\right ).
\end{equation}
Let $\omega _j \in \mascP _E(\rr {4d})$, $j=0,1,\dots ,N$, and suppose
\eqref{weightcond}$'$ holds. Then the map
\eqref{Weylmap}$'$ from $\maclS _{1/2}(\rr {2d}) \times \cdots \times
\maclS _{1/2}(\rr {2d})$ to $\maclS _{1/2}(\rr {2d})$
extends uniquely to a continuous and associative map from $\splM ^{p_1,q_1}
_{(\omega _1)}(\rr {2d}) \times \cdots \times \splM ^{p_N,q_N}
_{(\omega _N)}(\rr {2d})$ to $\splM ^{p_0',q_0'} _{(1/\omega _0)}(\rr {2d})$.
\end{tom}

\par

\begin{proof}
We may assume that $\masfR _N(\mathbf q')> 0$ since otherwise the
result follows from Proposition \ref{prop1}. By Lemma \ref{lemmaAthm0.3}
the conditions \eqref{pqconditions}$'$ imply
\begin{equation}\label{qjqj}
\masfR _N(\mathbf q') \le \min _{j=0,1,\dots ,N}\left ( \frac 1{q_j},\frac 1{q_j'}  \right ).
\end{equation}
Hence, if $r$ is defined by
$$
\frac 1r \equiv \masfR _N(\mathbf q'),
$$
then $r\ge 2$.

\par

By Proposition \ref{prop2} and \eqref{qjqj} we have
$$
\splM ^{u_1,q_1}_{(\omega _1)}\wpr \cdots \wpr \splM ^{u_N,q_N}_{(\omega _N)}
\subseteq \splM ^{u_0',q_0'}_{(1/\omega _0)},
$$
when $u_j \in [1,\infty]$ for $0 \le j \le N$ and
\begin{equation}\label{pqconditions2}
\frac 1r 
\le  \min _{j=0,1,\dots ,N}
\left ( \frac 1{u _j},\frac 1{u _j'},\frac 1{q_j'},
\masfR _N(\mathbf u )\right ).
\end{equation}
Due to Proposition \ref{p1.4} (2) the result follows if we can
prove that $p_j\le u_j$ for some $u_j\in [1,\infty ]$, $j=0,\dots ,N$,
that satisfy \eqref{pqconditions2}. We claim that
$u_j =\max (p_j ,r')$, $j=0, \dots ,N$, satisfy \eqref{pqconditions2}. 

\par

To wit, for such a choice we have
$$
\frac 1{u_j'}=\max \left ( \frac 1{p_j'},\frac 1r\right )\ge \frac 1r, \quad j=0,
\dots ,N, 
$$
and
\begin{equation}\label{rhoest}
\frac 1{u_j} = \min \left ( \frac 1{p_j},\frac 1{r'}\right )\ge \frac 1r,
\quad j=0, \dots ,N, 
\end{equation}
where \eqref{rhoest} follows from $r\ge 2$ and the
assumption $p_j\le r$. 

\par

Let $I$ be the set of all $j$ such that $r'\le p_j$. If
$I=\{ 0,1,\dots ,N\}$ the result follows from Proposition \ref{prop2}.
Therefore we may assume that there exists $k \in \{ 0,1,\dots ,N\}$
such that $k \notin I$.  Then $u_k=r'$, and \eqref{rhoest} gives
\begin{multline*}
(N-1) \masfR _N(\mathbf u ) =
\displaystyle {\frac 1{u_k} +\sum _{j \neq k}
\frac 1{u_j} -1}
=
\displaystyle {\frac 1{r'} + \sum _{j \neq k}
\frac 1{u_j} -1}
\\[1ex]
=
\displaystyle {-\frac 1{r} + \sum _{j \neq k}
\frac 1{u_j}}
\ge
\displaystyle {-\frac 1{r} + 
\frac N{r}} 
= \frac {N-1}r.
\end{multline*}
Hence
$$
\masfR _N(\mathbf u )\ge \frac 1r
$$
and the continuity assertion follows.

\par

The uniqueness and associativity follows from Proposition \ref{prop2} and
the inclusions above. 
\end{proof}

\par

In the next section we prove that Theorem \ref{thm0.3}$'$ is sharp for
$N=2$ with respect to the conditions on the Lebesgue exponents. On the
other hand, for $N\ge 3$, the result cannot be sharp. In fact, Theorem
\ref{thm0.3}$'$ with $N=2$ gives in particular that every unweighted
modulation space $\splM ^{p,q}$ is an $\splM ^{\infty ,1}$-module.
This property combined with the fact that $\splM ^{2,2}$ is an algebra
under the Weyl product give the inclusion
\begin{equation}\label{Ex3-lin-form}
\splM ^{\infty ,1}\wpr \splM ^{2,2}\wpr \splM ^{2,2} \subseteq \splM ^{2,2}.
\end{equation}
Theorem \ref{thm0.3}$'$ does however not contain this inclusion.

\par

The next result gives another sufficient condition for the map
\eqref{Weylmap}$'$ to be continuous that contains the inclusion
\eqref{Ex3-lin-form}. In the bilinear case $N=2$ the result follows from
Theorem \ref{thm0.3}$'$, because of the sharpness of the latter result
in that case.

\par

\begin{thm}\label{thm0.4}
Let $p_j,q_j\in [1,\infty ]$, $j=0,1,\dots , N$, and suppose
\begin{equation}\label{pqconditions3}
\masfR _N(\mathbf p) \ge 0
\quad \text{and}\quad
\frac 1{q_j'}\le  \frac 1{p_j} \leq \frac12.
\end{equation}
Let $\omega _j \in \mascP _E(\rr {4d})$, $j=0,1,\dots ,N$, and suppose
\eqref{weightcond}$'$ holds. Then the map
\eqref{Weylmap}$'$ from $\maclS _{1/2}(\rr {2d}) \times \cdots \times
\maclS _{1/2}(\rr {2d})$ to $\maclS _{1/2}(\rr {2d})$
extends uniquely to a continuous and associative map
from $\splM ^{p_1,q_1} _{(\omega _1)}(\rr {2d}) \times \cdots
\times \splM ^{p_N,q_N} _{(\omega _N)}(\rr {2d})$ to $\splM
^{p_0',q_0'} _{(1/\omega _0)}(\rr {2d})$.
\end{thm}

\par

The proof is by induction over $N$, and we need
the existence of certain intermediate weights.  
The following lemma guarantees the existence of such weights.

\par

\begin{lemma}\label{intermedWeight}
Let $\omega _0,\dots ,\omega _N\in \mascP _E(\rr {4d})$ satisfy
\eqref{weightcond}$'$. Then there exists
a weight $\vartheta \in \mascP _E(\rr {4d})$
such that
\begin{equation}\label{varthetaests}
\begin{aligned}
1 & \lesssim \frac{\omega _0(X_2+X_0,X_2-X_0) 
\omega _N(X_2+X_1,X_2-X_1)}{\vartheta (X_1+X_0,X_1-X_0)},
\quad X_0, X_1, X_2 \in \rr {2d}, 
\\[1ex]
1 & \lesssim \vartheta (X_{N-1}+X_0,X_{N-1}-X_0)\prod _{j=1}^{N-1}\omega
_j(X_j+X_{j-1},X_j-X_{j-1}), \\[1ex] 
& \qquad \qquad X_0, \dots, X_{N-1} \in \rr {2d}.
\end{aligned}
\end{equation}
\end{lemma}

\par

\begin{proof}
Let $X=X_{N-1}+X_0$, $Y=X_{N-1}-X_0$, define the linear mappings from
$\rr {6d}$ to $\rr {4d}$ given by
$$
T_{j,k} (X,Y,Z)= \left (  \frac {X+(-1)^jY}2 +Z, (-1)^k \left (   
\frac {X+(-1)^jY}2 -Z  \right )  \right ),
$$
for $j,k=1,2$, and set
\begin{align*}
H_1(X_1,\dots ,X_{N-2}) &\equiv \prod _{j=2}^{N-2}\omega _j
(X_j+X_{j-1},X_j-X_{j-1}),
\\[1ex]
H_2(X_1,X_{N-2}, X,Y)
&\equiv \omega _1 \big (T_{1,1}(X,Y, X_1) \big )\, 
\omega _{N-1}\big (T_{2,2}(X,Y,X_{N-2})\big )
\intertext{and}
H_3(X_N,X,Y)
&\equiv \omega _0\big( T_{1,1}(X,Y,X_N)
\big ) \omega _N\big (T_{2,1}(X,Y,X_N)\big ) .
\end{align*}
Then \eqref{weightcond}$'$ is equivalent to
\begin{equation*}
\left ( H_2\left (X_1,X_{N-2}, X , Y \right )
H_1(X_1,\dots ,X_{N-2}) \right )^{-1}
\lesssim
H_3 \left ( X_N,X,Y \right ).
\end{equation*}
The left hand side 
is independent of $X_N$
and the right hand side is independent of $X_1,\dots ,X_{N-2}$.

\par

If we define
$$
\vartheta (X,Y) \equiv \sup_{X_1,\dots ,X_{N-2} \in \rr {2d}} 
\left ( H_2\left (X_1,X_{N-2},X,
Y\right ) H_1(X_1,\dots ,X_{N-2}) \right )^{-1}
$$
then
\begin{multline*}
\vartheta (X_{N-1}+X_0,X_{N-1}-X_0) \lesssim H_3
\left ( X_N,X_{N-1}+X_0,X_{N-1} -X_0 \right ),
\\[1ex]
X_0,X_{N-1},X_N\in \rr {2d}, 
\end{multline*}
and \eqref{varthetaests} holds.

\par

It remains to show $\vartheta \in \mascP _E(\rr {4d})$. 
For $\ep >0$ arbitrary we have
\begin{multline*}
\vartheta (Z_1+Z_2,Y_1+Y_2)
\\[1ex]
\le \left ( H_2\left (X_1,X_{N-2},
Z_1+Z_2,Y_1+Y_2\right )
H_1(X_1,\dots ,X_{N-2}) \right )^{-1}+\ep
\end{multline*}
for some choice of $X_1,\dots ,X_{N-2}$. Since each $\omega _j$
is moderate we have
\begin{multline*}
\left ( H_2\left (X_1,X_{N-2},
Z_1+Z_2,Y_1+Y_2\right ) \right )^{-1}
\\[1ex]
\le C
\left ( H_2\left (X_1,X_{N-2},
Z_1,Y_1\right ) \right )^{-1}v(Z_2,Y_2),
\end{multline*}
for $C>0$ and some submultiplicative function $v \in L^\infty
_{\rm loc}(\rr {4d})$, which depends on
$\omega _1$ and $\omega _{N-1}$.

This estimate yields
\begin{multline*}
\vartheta (Z_1+Z_2,Y_1+Y_2)
\\[1ex]
\le C
\left ( H_2\left (X_1,X_{N-2},
Z_1,Y_1\right )
H_1(X_1,\dots ,X_{N-2}) \right )^{-1}v(Z_2,Y_2)+\ep
\\[1ex]
\le C \vartheta (Z_1,Y_1)v(Z_2,Y_2)+\ep .
\end{multline*}
Since $\ep >0$ is arbitrary and $C$ does not depend on $\ep$,
it follows that $\vartheta$ is $v$-moderate, and 
we may conclude that $\vartheta \in \mascP _E(\rr {4d})$. 
\end{proof}

\par

\begin{proof}[Proof of Theorem \ref{thm0.4}]
By Proposition \ref{p1.4} (2) we may assume $q_j'=p_j$, $j=0,\dots,N$.

\par

We start by proving the result for $N=2$. Assume 
\eqref{pqconditions3} for $N=2$. Then for every fixed $j\in \{ 0,1,2 \}$
we get
$$
\masfR _2(\mathbf p) = \masfR _2(\mathbf q') = \sum _{k=0}^2\frac 1{p_k} -1
\le \frac 1{p_j}.
$$
The continuity statement now follows from Theorem \ref{thm0.3}$'$.

\par

Next we perform the induction step. We assume that $N\ge 3$ and
the result holds for lower values of $N$. In particular we assume the
inclusion
\begin{equation}\label{induktionsantagande}
\splM ^{r_1,r_1'}_{(\omega_1)} \wpr \cdots \wpr \splM
^{r_{N-1},r_{N-1}'}_{(\omega _{N-1} )}
\subseteq \splM ^{r_0',r_0}_{(1/\vartheta)}
\end{equation}
whenever $r_j \geq 2$, 
$$
\sum _{j=0}^{N-1} \frac1{r_j} \geq 1, 
$$
and $(\vartheta,\omega _1, \dots, \omega _{N-1} ) \in \mascP
_E(\rr {4d})^{N}$ satisfy 
\eqref{weightcond}$'$.

We now distinguish two cases.

\par

In the first case we suppose that
\begin{equation}\label{ineq3}
\frac 1{p_0} + \frac 1{p_N} \le \frac 12 \quad \text{or}\quad
\frac 1{p_j}+\frac 1{p_{j+1}} \le \frac 12 \quad \text{for some $j \in
\{0,\dots ,N-1\}$.}
\end{equation}
By \eqref{duality0} and duality it suffices to consider the case when
the first inequality in \eqref{ineq3} holds. 
We define $r_0$ as
$$
\frac 1{r_0}= \frac 1{p_0}+\frac 1{p_N}\le \frac 12,
$$
and the result follows if we prove the inclusions
\begin{align}
\splM ^{p_1,p_1'}_{(\omega _1)}\wpr \cdots \wpr \splM ^{p_{N-1},p_{N-1}'}
_{(\omega _{N-1})} & \subseteq \splM ^{r_0',r_0}_{(1/\vartheta )}
\label{r0CaseN-1}
\intertext{and}
\splM ^{r_0',r_0}_{(1/\vartheta )}\wpr \splM ^{p_N,p_N'}_{(\omega _N)}
&\subseteq \splM ^{p_0',p_0}_{(1/\omega _0)},
\label{r0Case2}
\end{align}
where $\vartheta$ is chosen according to Lemma \ref{intermedWeight}.

\par

Since
$$
\frac 1{r_0}+\sum _{k=1}^{N-1}\frac 1{p_k} = \sum _{k=0}^{N}\frac 1{p_k}\ge 1
\quad \text{and}\quad \frac 1{r_0},\frac 1{p_j}\le \frac 12,\ j=1,\dots ,N-1,
$$
the inclusion \eqref{r0CaseN-1} follows from the induction
assumption \eqref{induktionsantagande}.

\par

By letting $s_0=p_0$, $s_1=r_0'$, $s_2=p_N$, it follows from the
choice of $r_0$ that $\masfR _2(\mathbf s)=0$, and the inclusion
\eqref{r0Case2} follows from Proposition \ref{prop1}, since
$(\omega_0, 1/\vartheta, \omega_N)$ satisfy
\eqref{weightcond}$'$ for $N=2$ by Lemma \ref{intermedWeight}.
The induction step is now complete in the first case by combining
\eqref{r0CaseN-1} and \eqref{r0Case2}.

\par

It remains to consider the second case where \eqref{ineq3} does
not hold. Therefore suppose that
\begin{equation}\label{ineq4}
\frac 1{p_0}+\frac 1{p_{N}} >\frac 12
\quad \text{and}\quad
\frac 1{p_j}+\frac 1{p_{j+1}}  > \frac 12, 
\end{equation}
for all $j=0,\dots,N-1$.
In particular we have 
$$
\frac1{p_0} + \frac1{p_N}  + \frac12 -1 >0,  
$$
so by the result for $N=2$ we have the inclusion
\begin{equation}\label{caseN2}
\splM ^{2,2}_{(1/\vartheta)} \wpr \splM ^{p_N,p_N'}
_{(\omega _N)}  \subseteq \splM ^{p_0',p_0}_{(1/\omega_0 )}. 
\end{equation}
Since $N \geq 3$, \eqref{ineq4} implies
\begin{equation}\nonumber
\frac 1{p_1} + \frac 1{p_2} + \cdots + \frac 1{p_{N-1}} + \frac
1 2 - 1  > 0, 
\end{equation}
and the induction hypothesis \eqref{induktionsantagande} thus
gives
\begin{equation}\label{r0CaseN-1a}
\splM ^{p_1,p_1'}_{(\omega _1)} \wpr \cdots
\wpr \splM ^{p_{N-1},p_{N-1}'} _{(\omega _{N-1})}
\subseteq \splM ^{2,2}_{(1/\vartheta )}.
\end{equation}
Combining \eqref{caseN2} and \eqref{r0CaseN-1a} gives the
induction step in the second case. The induction step is thus
complete so the continuity statement holds for any integer
$N \geq 2$.

\par

Finally, the uniqueness and associativity of the extension follows
as in the proof of Proposition \ref{prop1}. 
In fact, if $p_j=\infty$ then by the assumptions $q_j=1$, and a factor
$a_j\in \splM ^{\infty ,1}_{(\omega _j)}$ can be approximated narrowly
by elements in $\maclS _{1/2}$. If $p_j<\infty$ then the assumption
$2\le q_j'$ implies that a factor $a_j\in \splM ^{p_j ,q_j}_{(\omega _j)}$
can be approximated in norm by elements in $\maclS _{1/2}$.
\end{proof}

\par

We may use \eqref{calculitransform} and Proposition
\ref{propCalculiTransfMod} to extend Theorems
\ref{thm0.3}$'$ and \ref{thm0.4} to concern not only the Weyl product but general products
arising in the pseudo-differential calculi \eqref{e0.5} indexed by $t \in \re$. More precisely, for
every $t\in \mathbf R$, the $\wpr _t$ product with $N$ factors
\begin{equation}\label{tProdmap}
(a_1,\dots ,a_N)\mapsto a_1\wpr _t \cdots \wpr _t a_N
\end{equation}
from $\maclS _{1/2}(\rr {2d})\times \cdots \times
\maclS _{1/2}(\rr {2d})$ to
$\maclS _{1/2}(\rr {2d})$ is defined by the formula
$$
\op _t(a _1 \wpr _t \cdots \wpr _t a_N) = \op _t(a_1)\circ \cdots
\op _t(a_N).
$$
By \eqref{calculitransform} we have
\begin{multline*}
a _1 \wpr _t \cdots \wpr _t a_N
=
e^{it_0\scal {D_x}{D_\xi }} ((e^{-it_0\scal {D_x}{D_\xi }}a _1)
\wpr  \cdots \wpr  (e^{-it_0\scal {D_x}{D_\xi }} a_N)),
\\[1ex]
t_0=\frac 12 -t.
\end{multline*}
If we combine this relation with
Proposition \ref{propCalculiTransfMod}, Theorems
\ref{thm0.3}$'$ and \ref{thm0.4}, we get the following result.
The condition on the weight functions is
\begin{multline}\label{weightcondtcalc}
1 \lesssim \omega _0(T_t(X_N,X_0))\prod _{j=1}^N
\omega _j(T_t(X_{j},X_{j-1})),
%\\[1ex]
\quad X_0,\dots ,X_N \in \rr {2d},
\end{multline}
where
\begin{multline}\label{Ttdef}
T_t(x,\xi,y,\eta) = ( t x + (1-t)y,(1-t)\xi + t \eta ,\eta -\xi, x-y),
\\[1ex]
x,\xi, y,\eta \in \rr d.
\end{multline}

\par

\begin{thm}\label{thm0.3+0.4tOps}
Let $p_j,q_j\in [1,\infty ]$, $j=0,1,\dots , N$, be as in Theorems
\ref{thm0.3}$'$ or \ref{thm0.4}.
Let $t\in \mathbf R$, $\omega _j \in \mascP _E(\rr {4d})$, $j=0,1,\dots ,N$,
and suppose
\eqref{weightcondtcalc} and \eqref{Ttdef} hold. Then the map
\eqref{tProdmap} from $\maclS _{1/2}(\rr {2d}) \times \cdots \times
\maclS _{1/2}(\rr {2d})$ to $\maclS _{1/2}(\rr {2d})$
extends uniquely to a continuous and associative map from $M ^{p_1,q_1}
_{(\omega _1)}(\rr {2d}) \times \cdots \times M ^{p_N,q_N}
_{(\omega _N)}(\rr {2d})$ to $M ^{p_0',q_0'} _{(1/\omega _0)}(\rr {2d})$.
\end{thm}

\par

Finally we prove a continuity result for the twisted convolution. 
The map \eqref{Weylmap}$'$ is then replaced by
\begin{equation}\label{Twistmap}
(a_1,a_2,\dots ,a_N)\mapsto a_1*_\sigma a_2*_\sigma \cdots *_\sigma a_N.
\end{equation}
The following
result follows immediately from Theorem \ref{thm0.3}$'$ and Theorem \ref{thm0.4}.
Here the condition \eqref{weightcond}$'$ is replaced by
\begin{multline}\label{weightcond3}
1 \lesssim \omega _0(X_N-X_0,X_N+X_0)\prod _{j=1}^N
\omega _j(X_j-X_{j-1},X_j+X_{j-1}),
\\[1ex]
X_0,X_1,\dots ,X_N\in \rr {2d}.
\end{multline}

\par

\begin{thm}\label{thm0.3TwistConv}
Let $p_j,q_j\in [1,\infty ]$, $j=0,1,\dots , N$, and suppose that
\begin{equation*}
\max \left ( \masfR _N(\mathbf p') ,0 \right )
\le  \min _{j=0,1,\dots ,N} \left ( \frac 1{p_j'},\frac 1{q_j},
\masfR _N(\mathbf q)\right )
\end{equation*}
or
\begin{equation*}
\masfR _N(\mathbf q) \ge 0
\quad \text{and}\quad
\frac 1{p_j'}\le  \frac 1{q_j} \leq \frac12.
\end{equation*}
Suppose $\omega _j\in \mascP _E(\rr {4d})$, $j=0,1,\dots ,N$, satisfy
\eqref{weightcond3}. Then the map
\eqref{Twistmap} from $\maclS _{1/2}(\rr {2d}) \times \cdots \times
\maclS _{1/2}(\rr {2d})$ to $\maclS _{1/2}(\rr {2d})$
extends uniquely to a continuous and associative map from
$\splW^{p_1,q_1}_{(\omega _1)}(\rr {2d})
\times \cdots \times \splW^{p_N,q_N}_{(\omega _N)}(\rr {2d})$
to $\splW^{p_0',q_0'} _{(1/\omega _0)}(\rr {2d})$.
\end{thm}

\par

%%%%%%%%%%%%%%%%%%%%%%%%%%%%%%
\section{Necessary boundedness conditions}\label{sec3}
%%%%%%%%%%%%%%%%%%%%%%%%%%%%%%

\par

In this section we prove necessary conditions for continuity of the
Weyl product when $N=2$ and certain polynomially moderate weight
triplets that satisfy \eqref{weightcond}.

\par

More precisely, for weights of the form
\begin{equation}\label{weightcond2}
\begin{aligned}
\omega _0(X,Y) &= \frac {\vartheta _0(X+Y)}{\vartheta _2(X-Y)},\quad
\omega _1(X,Y)   = \frac {\vartheta _2(X-Y)}{\vartheta _1(X+Y)},
\\[1ex]
\omega _2(X,Y) &= \frac {\vartheta _1(X-Y)}{\vartheta _0(X+Y)},\quad
\end{aligned}
\end{equation}
where $\vartheta _j \in  \mascP (\rr {2d})$, $j=0,1,2$, we have
the following result. Note that the necessary condition
\eqref{partialineq} equals the sufficient condition
\eqref{pqconditions} of Theorem \ref{thm0.3}. 

\par

\begin{thm}\label{thm0.3Rev}
Let $p_j,q_j\in [1,\infty ]$, $\vartheta _j\in \mascP (\rr {2d})$,
and define $\omega _j\in \mascP (\rr {4d})$ by
\eqref{weightcond2} for $j=0,1,2$. If the map
\eqref{Weylmap} 
from $\maclS _{1/2}(\rr {2d}) \times \maclS _{1/2}(\rr {2d})$ to
$\maclS _{1/2}(\rr {2d})$ is extendable to a continuous map from
$\splM ^{p_1,q_1}_{(\omega _1)}(\rr {2d})
\times \splM ^{p_2,q_2}_{(\omega _2)}(\rr {2d})$ to $\splM ^{p_0',q_0'}
_{(1/\omega _0)}(\rr {2d})$,
then
\begin{equation}\label{partialineq}
\max(\masfR _2(\mathbf q'),0)
\le  \min _{j=0,1,2} \left ( \frac {1}{p_j},\frac {1}{q_j'}, \masfR
_2(\mathbf p)\right ).
\end{equation}
\end{thm}

\par

\begin{rem}
The conditions \eqref{weightcond2} on the weights appear naturally when
Weyl operators with symbols in modulation spaces act on modulation spaces. 
For example, if \eqref{weightcond2} is fulfilled,
$p,q\in [1,\infty ]$ and $a_1\in \splM ^{\infty ,1}_{(\omega _1)}(\rr {2d})$,
then $\op ^w(a_1)$ is continuous from $M^{p,q}_{(\vartheta _1)}(\rr d)$
to $M^{p,q}_{(\vartheta _{2})}(\rr d)$ (cf. \cite[Theorem 6.2]{Toft8}).
\end{rem}

\par

We need some preparations for the proof. The first result is a
reduction to trivial weights. For $\omega \in \mascP (\rr {2d})$
we use the notation $S_{(\omega)}(\rr {2d})$ for the symbol
space of all $a \in C^\infty(\rr {2d})$ such that $(\partial^\alpha
a)/\omega \in L^\infty$ for all $\alpha \in \nn {2d}$. 

\par

\begin{lemma}\label{WeightReduction}
Let $\vartheta, \vartheta _1,\vartheta _2\in \mascP (\rr {2d})$, $\omega _1,
\omega _2\in \mascP (\rr {4d})$ and let $p,q\in [1,\infty ]$. 
There exist
$b_j\in S_{(\vartheta _j)}(\rr {2d})$ and $c_j\in S_{(1/\vartheta _j)}
(\rr {2d})$ such that
$$
\op ^w(b_j)\circ \op ^w(c_j) = \op ^w(c_j)\circ \op ^w(b_j)
$$
is the identity map on $\mascS '(\rr d)$, for $j=1,2$,  
and the following holds:
\begin{enumerate}
\item[{\rm{(1)}}] $\op ^w(b_j)$ is continuous and bijective from
$M ^{p,q}_{(\vartheta )} (\rr d)$ to $M ^{p,q}_{(\vartheta /
\vartheta _j)}(\rr d)$ with inverse $\op ^w(c_j)$, $j=1,2$;

\vrum

\item[{\rm{(2)}}] if $\omega _2(X,Y)\lesssim \omega _1(X,Y)
/\vartheta (X+Y)$, then the map
\eqref{Weylmap} on $\mascS (\rr {2d})$ extends uniquely to a
continuous map from $\splM ^{p,q}_{(\omega _1)}(\rr {2d})
\times S_{(\vartheta )}(\rr {2d})$ to $\splM ^{p,q}_{(\omega
_2)}(\rr {2d})$;

\vrum

\item[{\rm{(3)}}] if $\omega _2(X,Y)\lesssim \omega _1(X,Y)/
\vartheta (X-Y)$, then the map \eqref{Weylmap} on $\mascS
(\rr {2d})$ extends uniquely to a continuous map from
$S_{(\vartheta )}(\rr {2d})\times \splM ^{p,q}_{(\omega _1)}(\rr {2d})$
to $\splM ^{p,q}_{(\omega _2)}(\rr {2d})$;

\vrum

\item[{\rm{(4)}}] if $\omega (X,Y)=\vartheta _2(X-Y)/\vartheta
_1(X+Y)$, then the map $a\mapsto b_2 \wpr a \wpr c_1$ is
continuous on $\mascS (\rr {2d})$, and extends uniquely to
a continuous and bijective map from $\splM ^{p,q}
_{(\omega )}(\rr {2d})$ to $\splM ^{p,q}(\rr {2d})$.
\end{enumerate}
\end{lemma}

\par

\begin{proof}
The assertion (1) follows immediately from
\cite[Theorem 3.1]{GrochToft1}. In order to prove (2) and (3), we
first use the assumption that $\omega _1$ and $\vartheta$ are
moderate, which gives
\begin{alignat*}{2}
\omega _1(X,Y) &\lesssim \omega _1(X-Y+Z,Z)\eabs {Y-Z}^N,
\\[1ex]
\omega _1(X,Y) &\lesssim \omega _1(X+Z,Y-Z)\eabs {Z}^N, 
\\[1ex]
\frac 1{\vartheta (X+Y)} &\lesssim \frac {\eabs {Y-Z}^N}{\vartheta (X+Z)}, 
\qquad \frac 1{\vartheta (X-Y)} \lesssim
\frac {\eabs {Z}^N}{\vartheta (X-Y+Z)},
\end{alignat*}
for some $N \geq 0$. The assumption in (2) leads to
\begin{align}
\omega _2(X,Y) &\lesssim \omega _1(X-Y+Z,Z) \,
\frac {\eabs {Y-Z}^{2N}}{\vartheta (X+Z)},\label{Weights(2)}
\intertext{and the assumption in (3) gives}
\omega _2(X,Y) &\lesssim \frac {\eabs {Z}^{2 N}}{\vartheta (X-Y+Z)} \,
\omega _1(X+Z,Y-Z). \label{Weights(3)}
\end{align}
If we set $\omega (X,Y) = \eabs Y^{2 N}/\vartheta (X)$ then
Theorem \ref{thm0.3} with \eqref{Weights(2)} and
\eqref{Weights(3)}, respectively, shows that the map \eqref{Weylmap}
from $\mascS \times \mascS$ to $\mascS$ extends uniquely
to a continuous map from $\splM ^{p,q}_{(\omega _1)}\times
\splM ^{\infty ,1}_{(\omega )}$ and $\splM ^{\infty ,1}_{(\omega )}\times
\splM ^{p,q}_{(\omega _1)}$, respectively, to $\splM ^{p,q}_{(\omega _2)}$.
 The assertions (2) and (3) now follow from
$S_{(\vartheta )}\subseteq \splM ^{\infty ,1}_{(\omega)}$, which
is a consequence of \cite[Proposition 2.7 (3)]{GrochToft1}.

\par

Finally (4) follows by a straight-forward combination of (1)--(3).
\end{proof}

\par

Lemma \ref{WeightReduction} shows that for weights $\omega_j$,
$j=0,1,2$, satisfying \eqref{weightcond2} we have
\begin{equation*}
\| a \wpr b \|_{\splM _{(1/\omega_0)}^{p_0',q_0'}}
\lesssim \| a \|_{\splM _{(\omega_1)}^{p_1,q_1}} \| b \|_{\splM
_{(\omega_2)}^{p_2,q_2}}, \quad a,b \in \maclS _{1/2}(\rr {2d}), 
\end{equation*}
if and only if
\begin{equation*}
\| a \wpr b \|_{\splM ^{p_0',q_0'}}
\lesssim \| a \|_{\splM ^{p_1,q_1}} \| b \|_{\splM ^{p_2,q_2}},
\quad a,b \in \maclS _{1/2}(\rr {2d}), 
\end{equation*}
thus reducing the problem to the case of trivial weights.

\par

It remains to show Theorem \ref{thm0.3Rev} with trivial weights.
By the last part of Lemma \ref{lemmaAthm0.3}, the
Condition \eqref{partialineq} is equivalent to the inequalities
\begin{align}
1 & \leq \frac1{p_0} + \frac1{p_1} + \frac1{p_2},
\label{suff1}
\\[1ex]
1 & \le \frac1{q_j}  +  \frac1{q_k}, \quad 0 \le j \neq k \le 2, 
\label{suff3}
\\[1ex]
2 - \frac1{q_0} - \frac1{q_1} - \frac1{q_2} & \leq \frac1{p_j},
\quad j=0,1,2,
\label{suff4}
\\[1ex]
3 & \leq \frac1{q_0} + \frac1{q_1} + \frac1{q_2} + \frac1{p_0} +
\frac1{p_1} + \frac1{p_2},
\label{suff5}
\end{align}
which we prove in the following sequence of lemmas.

\par

The first lemma shows that \eqref{suff1} and \eqref{suff3} are necessary
for the requested continuity.

\par

\begin{lemma}\label{gauss}
Let $1\leq p_j,q_j\leq\infty$ for $j=0,1,2$, and suppose
\begin{equation}\label{conttrivweight}
\|a\wpr b\|_{\splM ^{p'_0,q'_0}} \lesssim \|a\|_{\splM ^{p_1,q_1}}
\nm b{\splM ^{p_2,q_2}} \quad \text{for every}\quad
a,b\in \maclS _{1/2}(\rr {2d}).
\end{equation}
Then \eqref{suff1} and \eqref{suff3} hold.
\end{lemma}

\par

\begin{proof}
First we observe that the assumption \eqref{conttrivweight}
combined with \eqref{duality0}, duality and $\overline{a \wpr
b} = \overline{b} \wpr \overline{a}$ (cf. \cite{HTW})  imply that 
\begin{equation*}
\nm {a\wpr b}{\splM ^{r'_0,s'_0}} \lesssim \nm a{\splM
^{r_1,s_1}} \nm b {\splM ^{r_2,s_2}}
\quad \text{for every}\quad
a,b\in \maclS _{1/2}(\rr {2d}),
\end{equation*}
when $r_j=p_{\sigma (j)}$ and $s_j=q_{\sigma (j)}$, where
$\sigma$ is any permutation $\{0,1,2\}$.
By \cite[Corollary 3.4]{HTW} we therefore have
\begin{equation*}
1 \leq \frac1{q_j} + \frac1{q_k}, \quad 0 \leq j \neq k \leq 2,
\end{equation*}
i.{\,}e. \eqref{suff3}.

\par

In order to show \eqref{suff1}, we consider the family of functions $a_\lambda(X)=e^{-\lambda|X|^2}$,
$X\in\rdd$, $\lambda>0$. Straight-forward computations show that (cf.
\cite[Proposition 3.1]{HTW})
\begin{equation}\label{dilgausmod}
\|a_\lambda\|_{\splM ^{r,s}}^{1/d}=\pi^{\frac1r+\frac1s-1}r^{-1/r}
s^{-1/s}\lambda^{-1/r}(1+\lambda)^{1/r+1/s-1}
\end{equation}
and
\begin{equation}\label{gauswpr}
a_\lambda\wpr a_\lambda(X)=(1+\lambda^2)^{-d}\exp \left(-\frac{2\lambda}
{1+\lambda^2}|X|^2\right).
\end{equation}
This gives
\begin{multline}\label{normpqwprgauss}
\|a_\lambda\wpr a_\lambda\|
_{\splM ^{r,s}}^{1/d}
\\[1ex]
=
\pi^{1/r+1/s-1}r^{-1/r} s^{-1/s}
(1+\lambda^2)^{-1/s} (2\lambda)^{-1/r}(1+\lambda)^{2(1/r+1/s-1)}.
\end{multline}
From the assumption \eqref{conttrivweight} we obtain
\begin{multline*}
\pi^{1/{p'_0}+1/{q'_0}-1}{p'_0}^{-1/{p'_0}}
{q'_0}^{-1/{q'_0}} (1+\lambda^2)^{-1/{q'_0}}(2\lambda)
^{-1/{p'_0}}(1+\lambda)^{2(1/{p'_0}+1/{q'_0}-1)}\quad\quad\quad
\\[1ex]
\leq C \pi^{1/{p_1}+1/{q_1}-1}{p_1}^{-1/{p_1}}
{q_1}^{-1/{q_1}} \lambda^{-1/{p_1}}(1+\lambda)^{1/{p_1}
+ 1/{q_1}-1}
\\[1ex]
\times \pi^{1/{p_2}+
1/{q_2}-1}{p_2}^{-1/{p_2}} {q_2}^{-1/{q_2}} \lambda^{-1/{p_2}}
(1+\lambda)^{1/{p_2}+1/{q_2}-1}.
\end{multline*}
Letting $\lambda\to0$ gives the inequality \eqref{suff1}.
\end{proof}

\par

Next we introduce more general Gaussians of the form
\begin{equation}\label{alambdamu}
a_{\lambda ,\mu}(x,\xi ) =e^{-\lambda |x^2|-\mu |\xi |^2},
\quad x,\xi \in \rr d, \quad \lambda,\mu>0.
\end{equation}
We consider $a_{\lambda _1,\mu _1}\wpr a_{\lambda _2,\mu _2}$,
where $\lambda _j,\mu _j >0$ satisfy
\begin{equation}\label{mulambdacond}
\frac {\lambda _1}{\mu _1} = \frac {\lambda _2}{\mu _2},
\end{equation}
so that
\begin{equation}\label{varthetacond}
\nu \asymp \nu _0\quad \text{when}\quad
\nu = 1+\lambda _1\mu _2 = 1+\lambda _2\mu _1,\
\nu _0= 1+\lambda _1\mu _2 +\lambda _2\mu _1.
\end{equation}

\par

The first part of the following result follows by straight-forward
computations and \eqref{tvist1}.
The other statements follow from the first part and
\cite[Lemma 1.8]{Toft2}.

\par

\begin{lemma}\label{lemmamixedgaussian}
Let $r,s\in [1,\infty]$, let $\lambda, \mu, \lambda _j,\mu _j >0$, $j=1,2$, satisfy
\eqref{mulambdacond}, and define $\nu$ and $\nu _0$
by \eqref{varthetacond}.
Then
$$
a_{\lambda _1,\mu _1}\wpr a_{\lambda _2,\mu _2} (x,\xi) =
(2 \pi)^{-d/2} \nu ^{-d}e^{-(\lambda _1+\lambda _2)|x|^2/\nu }
e^{-(\mu _1+\mu _2)|\xi |^2/\nu },
$$
$$
\nm {a_{\lambda ,\mu}}{\splM ^{r,s}} ^{1/d}= c_{r,s}(\lambda \mu)^{-1/2r}\big (
(1+\lambda )(1+\mu )  \big ) ^{(1/r+1/s-1)/2} ,
$$
and
\begin{multline*}
\nm {a_{\lambda _1,\mu _1}\wpr a_{\lambda _2,\mu _2}}
{\splM ^{r,s}} ^{1/d}
\\[1ex]
=C_{r,s} \nu ^{-1/s}\big (  (\lambda _1+\lambda _2)(\mu _1
+\mu _2) \big )^{-1/(2r)} \big (  (\nu +\lambda _1+\lambda
_2)(\nu +\mu _1 +\mu _2) \big )^{(1/r + 1/s -1)/2},
\end{multline*}
for some constants $c_{r,s}, C_{r,s} >0$ which only depend on $r,s$.
\end{lemma}

\par

Lemma \ref{lemmamixedgaussian} is used in the proof of the 
following result, which shows that \eqref{suff5} is a consequence of 
the requested continuity.

\par

\begin{lemma}\label{LemmaSuff5}
If \eqref{conttrivweight} holds then \eqref{suff5}
is true.
\end{lemma}

\par

\begin{proof}
Let $\lambda _1=\lambda _2=1/\mu _1=1/\mu _2 =\lambda >1$. Then
$\nu  = 2$, and Lemma \ref{lemmamixedgaussian} gives
\begin{equation*}
\nm {a_{\lambda _1,\mu _1}\wpr a_{\lambda _2,\mu _2}}
{\splM ^{p_0',q_0'}} \asymp \lambda ^{d(1/p_0' +1/q_0'-1)/2},
\end{equation*}
and
$$
\nm {a_{\lambda _j,\mu _j}}{\splM ^{p_j,q_j}} \asymp \lambda
^{d(1/p_j +1/q_j-1)/2}, \quad j=1,2.
$$
The assumption \eqref{conttrivweight} together with $\lambda
\rightarrow +\infty$ give
$$
\frac 1{p_0'}+\frac 1{q_0'} -1 \le \frac 1{p_1}+\frac 1{q_1} -1 
+  \frac 1{p_2}+\frac 1{q_2} -1,
$$
which is the same as \eqref{suff5}.
\end{proof}

\par

Finally we show that \eqref{conttrivweight} implies \eqref{suff4}.

\par

\begin{lemma}\label{LemmaSuff4}
If \eqref{conttrivweight} holds then \eqref{suff4} is true.
\end{lemma}

\par

\begin{proof}
By duality it suffices to show that \eqref{conttrivweight} implies that
\begin{equation*}
\frac1 {p_0'} + \frac 1 {q_0'} \leq \frac 1 {q_1} + \frac 1 {q_2} \qquad
\mbox{or} \qquad \frac1 {p_2'} + \frac 1 {q_0'} \leq \frac 1 {q_1} +
\frac 1 {q_2}.
\end{equation*}
The proof is a contradictary argument. We assume that
\eqref{conttrivweight} holds,
\begin{equation}\label{assump2}
\frac1 {p_0'} + \frac 1 {q_0'} > \frac 1 {q_1} + \frac 1 {q_2}
\end{equation}
and
\begin{equation}\label{assump3}
\frac1 {p_2'} + \frac 1 {q_0'} > \frac 1 {q_1} + \frac 1 {q_2}.
\end{equation}
This will lead to a contradiction which shows that
\eqref{conttrivweight} implies
\begin{equation*}
\frac1 {p_0'} + \frac 1 {q_0'} \leq \frac 1 {q_1} + \frac 1 {q_2},
\end{equation*}
i.e. \eqref{suff4} with $j=0$. The cases $j=1,2$ follows by duality.

\par

Thus we assume \eqref{assump2}, \eqref{assump3} and
\begin{equation}\label{inklusion3}
\splM ^{p_1,q_1} (\rr {2d}) \wpr \splM ^{p_2,q_2} (\rr {2d})
\subseteq \splM ^{p_0',q_0'} (\rr {2d}).
\end{equation}
From \eqref{assump3} we may conclude that there exists $\ep>0$ such that
\begin{equation}\label{expvillkor}
\frac1{p_2} + \frac1{q_2} + \frac1{q_1} - \frac1{q_0'}  < 1 - \frac{4 \ep}{d}.
\end{equation}

\par

The rest of the proof is an adaptation of the proof of
\cite[Theorem 3.6]{HTW} (see also the proof of
\cite[Theorem 5]{Grochenig1c}).

\par

Let $0\le g\in C_0^\infty (\rr d)\setminus 0$ be supported in
a ball with center in the origin and radius $1/4$.
For $n\in\mathbf Z^d$ we  set
\begin{equation}\label{mathsfd}
\mathsf d_n=\mathsf d_{n,\ep }=\left\{
\begin{array}{ll}
  1&n=0
\\[1ex]
  |n|^{-(d+{\ep })}&n\ne 0\,,
\end{array}\right.
\end{equation}
so that $\{ \mathsf d_n\} \in l^1$.
We also set
\begin{equation*}
{\alpha} _n = \mathsf d_n^{1/p_2}, \qquad  {\beta} _n =
\mathsf d_n^{1/q_1}, \qquad
{\gamma} _n = \mathsf d_n^{1/q_2}, \qquad
{\eta} _n = \mathsf d_n^{1/q_0},
\end{equation*}
and we let $\tau _n$ be the operator $\tau _nf = f(\cdot -n)$.
Our plan is to use the family of functions on $\rr d$ 
\begin{equation}\label{f1234def}
\begin{alignedat}{2}
f_1 &= \sum _{n}\alpha _n\tau _ng, & \qquad f_2=f_{2,N} &= \sum
_{|n|\le N}\beta _n\tau _ng,
\\[1ex]
f_3 &= \sum _{n} \gamma _n\tau _ng, & \qquad f_4 &= \sum
_{n}\eta _n\tau _ng,
\end{alignedat}
\end{equation}
to construct an element $b\in \splM ^{p_2,q_2}(\rr {2d})$ and
a sequence $\{ a_N\}$ in $\mathscr S(\rr {2d})$ such that
$\{ a_N\}$ is uniformly bounded in $\splM ^{p_1,q_1} (\rr {2d})$
but $\{ a_N\wpr b\}$ is not a
bounded sequence in $\splM ^{p_0',q_0'} (\rr {2d})$. This is the
desired contradiction to \eqref{inklusion3}.

\par

By \cite[Remark 1.3]{HTW} we know that the sequence $\{ f_{2,N}\}
\subseteq \mathscr S(\rr d)$ is uniformly bounded in $M^{q_1,1} (\rr {d})$,
and that
\begin{align*}
& f_1\in M^{p_2,1} (\rr {d}),\quad \widehat f_3 \in \mathscr
FM^{q_2,1} (\rr {d}) \subseteq M^{1,q_2} (\rr {d}),
\\[1ex]
& f_4 \in M^{q_0,1} (\rr {d}) \subseteq
M^{q_0} (\rr {d}).
\end{align*}

\par

In a moment we will prove that if we choose $\fy \in C_0^\infty (\rr
d)\setminus 0$  and define $a_N$ and $b$ as
\begin{equation}\label{aNbdef}
a_N=W_{\fy ,f_{2,N}} \quad \text{and}\quad \op ^w(b)h=f_1\cdot
(f_3*h),\quad h\in C_0^\infty (\rr d),
\end{equation}
then
\begin{equation}\label{abegensk}
\begin{aligned}
\nm {a_N}{\splM ^{p_1,q_1}} &\le C, \qquad b
\in {\splM ^{p_2,q_2}} (\rr {2d}) \quad
\text{and}
\\[1ex]
\op ^w(b)f_4 &= \sum _{n} \lambda _{n}\tau _ng_0,\quad
\text{where} \quad  g_0 = g\cdot (g*g),
\\[1ex]
\text{and}\qquad \lambda _n &\ge C |n|^{-d(1/q_2+1/p_2-1/q_0')
-\ep (1/q_2+1/p_2+1/q_0)},
\end{aligned}
\end{equation}
for some constant $C>0$ which is independent of $N$.
We note that $g * g$ is supported in a ball with center at the
origin and radius $1/2$.

\par

Assuming this for a while we may proceed as follows. From
\eqref{inklusion3} and \eqref{abegensk} we get that $(a_N\wpr b)$
is a bounded sequence in $\splM ^{p_0',q_0'} (\rr {2d})$, which
implies that $\op ^w(a_N\wpr b)$ is a uniformly bounded sequence
of continuous operators from $M^{q_0}(\rr {d})$ to
$M^{q_0'}(\rr {d})$. In fact,  \eqref{suff3} gives $2/q_0' \le
1/q_1+1/q_2$ which combined with \eqref{assump2} yield
$1/q_0' < 1/p_0'$.  Now \cite[Theorem 7.1]{Grochenig1b} or
\cite[Theorem 4.3]{Toft2} together with Proposition \ref{p1.4}
give the assertion. (See also \cite[Theorem 1.1]{CTW}.)

\par

On the other hand, since $f_4\in M^{q_0} (\rr {d})$ and
$\op ^w(b)f_4=f_1\cdot (f_3*f_4)$, we get
\begin{equation*}
\op ^w(a_N\wpr b)f_4 = (\op ^w(W_{\fy ,f_2}))(f_1\cdot (f_3*f_4)),
\end{equation*}
which by \cite[Lemma 1.6]{HTW} gives
\begin{equation}\label{aNbweylrel}
\op ^w(a_N\wpr b)f_4 =(2 \pi)^{-d/2}(f_1\cdot (f_3*f_4),f_2)\fy .
\end{equation}
Now \eqref{aNbdef}, \eqref{abegensk} and the fact that $(g_0,g)>0$
show that %%
\begin{align*}
(f_1\cdot (f_3*f_4),f_2) & \ge C\sum _{|n|\le N}\lambda _n\beta _n \\
& \ge C' \sum _{|n|\le N }|n|^{-d(1/q_1+1/q_2+1/p_2-1/q_0')-\ep (1/q_1
+1/q_2+1/p_2+1/q_0)},
\end{align*}
which gives, using \eqref{expvillkor},
\begin{equation}\label{f1234uppsk}
(f_1\cdot (f_3*f_4),f_2)\ge C'\sum
_{|n|\le N } |n|^{-d}.
\end{equation}
Consequently, since the right-hand side can be made arbitrarily large
by increasing $N$, we have obtained a contradiction to the uniformly
boundedness of $\op ^w(a_N\wpr b)$ as a sequence of operators from
$M^{q_0}(\rr {d})$ to $M^{q_0'}(\rr {d})$. Hence our assumption,
contrary to the statement, is wrong, and the result follows.

\par

It remains to prove \eqref{abegensk}. From the assumptions we have
that $\fy \in C_0^\infty (\rr {d}) \subseteq M^1(\rr {d})$ and $f_2\in
M^{q_1,1} (\rr {d})$. From \cite[Theorem 4.1]{Toft2} it follows that
$a_N=W_{\fy ,f_2}$ is uniformly bounded in $M^{1,q_1} (\rr {2d})$,
and likewise in $M^{p_1,q_1} (\rr {2d})$. We have $\op ^w(b) = \op
_0(c)$ with $c=f_1\otimes \widehat f_3$ and $\op_0(c)$ is a
pseudo-differential operator of 
Kohn--Nirenberg type, i.{\,}e. $t=0$ in \eqref{e0.5}. 
Since $f_1,\widehat f_3\in M^{p_2,q_2} (\rr {d})$, it
follows that $c\in M^{p_2,q_2} (\rr {2d})$. 
By \cite[Remark 1.5]{HTW} we then obtain $b \in M^{p_2,q_2} (\rr {2d})
={\splM ^{p_2,q_2}} (\rr {2d})$. 

\par

In order to prove the last part of \eqref{abegensk} we note that
$$
f_3*f_4 =\sum _n\mu _n \tau _n (g*g),
$$
where $\{ \mu _n\}$ is the discrete convolution between $\{ \gamma _n\}$ and
$\{ \eta _n\}$, i.{\,}e.
$$
\mu _n = \sum _k\gamma _{n-k}\eta _k.
$$
By Young's inequality it follows that $(\mu _n)\in l^{r}$, where
$1/q_2+1/q_0=1+1/r$. Here \eqref{suff3} guarantees that $r \in [1,\infty]$.

\par

From the support
properties of $g$ and $g*g$, it follows that
$$
f_1\cdot (f_3*f_4) =\sum _n \lambda _n\tau _ng_0
$$
where $\lambda _n=\alpha _n\mu _n$.
We have to estimate $\lambda _n$. For any $n\in \zz d$, let
$$
\Omega _n = \sets {k\in \zz d}{|k|\le |n|,\ k\neq 0,\ k\neq n}.
$$
We have
\begin{multline*}
\mu _n = \sum _k \gamma_{n-k}\eta_k
\geq \sum_{k\in \Omega _n}
|n-k|^{-{(d+\ep )}/{q_2}}|k|^{-{(d+\ep )}/{q_0}}
\\[1ex]
\geq C |n|^{d(1-1/q_2-1/q_0) -
\ep ( 1/q_2+1/{q_0})},
\end{multline*}
for some $C>0$. Hence
\begin{multline*}
\lambda _n =\alpha _n \mu _n
\ge C|n|^{-(d+\ep)/p_2} |n|^{d(1-1/q_2-1/q_0) -
\ep ( 1/q_2+1/{q_0})}
\\[1ex]
= C |n|^{-d(1/q_2+1/p_2-1/q_0')-\ep (1/q_2+1/p_2+1/q_0)}.
\end{multline*}
This proves \eqref{abegensk} and the result follows.
\end{proof}

\begin{proof}[Proof of Theorem \ref{thm0.3Rev}]
By Lemma \ref{WeightReduction} we may assume trivial weights.
By Lemmas \ref{gauss}, \ref{LemmaSuff5} and \ref{LemmaSuff4},
the inequalities \eqref{suff1}--\eqref{suff5} hold. Thus Lemma
\ref{lemmaAthm0.3} shows that \eqref{partialineq} holds true.
\end{proof}

%%%%%%%%%%%%%%%%%%%%%%%%%%%%%
\section{Some particular cases}\label{sec4}
%%%%%%%%%%%%%%%%%%%%%%%%%%%%%

\par

We list in a table some special cases of the inclusion results
for the Weyl product on unweighted modulation spaces.
More precisely, we compare the inclusion results
in Theorem \ref{thm0.3} and \cite[Theorem 0.3$'$]{HTW}
when the exponents belong to $p_j,q_j\in \{ 1,2,\infty \}$.
The table illustrates the generality of Theorem \ref{thm0.3} as
compared to \cite[Theorem 0.3$'$]{HTW}. In fact the latter
result gives no inclusions in modulation spaces for several of
the studied cases.

\par

On the other hand \cite[Theorem 0.3$'$]{HTW} combined
with Proposition \ref{p1.4} (2) gives the inclusions in
Theorem \ref{thm0.3} for these particular cases.

\par

A corresponding table can be made for the twisted
convolution acting on Wiener amalgam spaces
$\splW ^{p,q}_{(\omega )}$, provided the involved Lebesgue
exponents $p_j,q_j$ are interchanged
(cf. Theorem \ref{thm0.3TwistConv}). 
In the literature it is common that the norm in the Wiener amalgam space
$W^{p,q}_{(\omega )}$ is defined with reversed order of the Lebesgue
exponents, i.{\,}e. the norm is defined by
$$
\nm f {W^{p,q}_{(\omega )}} \equiv \Big (\int _{\rr d}\Big (\int _{\rr d}
|V_\phi f(x,\xi )\omega (x,\xi )|^p\, d\xi \Big )^{q/p}\, dx \Big )^{1/q},
$$
instead of \eqref{modnorm2} (cf. \cite{Feichtinger1, Feichtinger2,
Feichtinger6}).

\par

With this definition of the Wiener amalgam spaces,
the table for the twisted convolution
is the same as the table below, if the Weyl product $\wpr$
is replaced by the twisted convolution $*_\sigma$, and 
the modulation spaces $\splM ^{p,q}$ are replaced by $\splW ^{p,q}$

\par

Finally we remark that the relations in the table are valid for 
weighted spaces provided the involved weights satisfy \eqref{weightcond}.

\par 

\newpage

{\begin{tabular}{c||c|c|c}
& & Theorem \ref{thm0.3}: &  \cite[Theorem 0.3$'$]{HTW}: 
\\[2ex]
    \textbf{No.} & $\mathbf{\splM ^{p_1,q_1}\wpr \splM ^{p_2,q_2}}$
    & $\mathbf{\splM ^{p_0',q_0'}}$ 
     & $\mathbf{\splM ^{p_0',q_0'}}$ 
\\[2ex]
\hline
                    & & &
\\
     \textbf{1} &  $\splM ^{1,1}\wpr \splM ^{1,1}$ & $\splM ^{1,1}$ &  ---  
\\[2ex]
     \textbf{2} &  $\splM ^{1,1}\wpr \splM ^{1,2}$ & $\splM ^{1,2}$
     &  ---  
\\[2ex]
     \textbf{3} &  $\splM ^{1,1}\wpr \splM ^{1,\infty}$
     & $\splM ^{1,\infty}$ &  ---  
\\[2ex]
     \textbf{4} &  $\splM ^{1,1}\wpr \splM ^{2,1}$ & $\splM ^{1,1}$ &  ---
     
\\[2ex]
     \textbf{5} &  $\splM ^{1,1}\wpr \splM ^{2,2}$ & $\splM ^{1,2}$
     &  --- 
\\[2ex]
     \textbf{6} &  $\splM ^{1,1}\wpr \splM ^{2,\infty}$ & $\splM ^{1,\infty}$
     &  ---  
\\[2ex]
     \textbf{7} &  $\splM ^{p,q}\wpr \splM ^{\infty,1}$
     & $\splM ^{p,q}$ &  $\splM ^{p,q}$ 
\\[2ex]
     \textbf{8} &  $\splM ^{1,1}\wpr \splM ^{\infty ,2}$
     & $\splM ^{1,2}$ &  $\splM ^{1,2}$
\\[2ex]
     \textbf{9} &  $\splM ^{1,1}\wpr \splM ^{\infty ,\infty}$
     & $\splM ^{1,\infty}$ &  $\splM ^{1,\infty}$ 
\\[2ex]
     \textbf{10} &  $\splM ^{2,2}\wpr \splM ^{2,2}$ & $\splM ^{2,2}, \splM ^{1,\infty}$
     & $\splM ^{2,2}, \splM ^{1,\infty}$
\\[2ex]
     \textbf{11} &  $\splM ^{1,2}\wpr \splM ^{1,2}$ & $\splM ^{2,2}$
     &  ---  
\\[2ex]
     \textbf{12} &  $\splM ^{1,2}\wpr \splM ^{2,1}$ & $\splM ^{1,2}$
     &  ---  
\\[2ex]
     \textbf{13} &  $\splM ^{1,2}\wpr \splM ^{2,2}$
     & $\splM ^{2,2}$, $\splM ^{1,\infty}$ 
     &  ---  
\\[2ex]
     \textbf{14} &  $\splM ^{1,2}\wpr \splM ^{\infty ,2}$
     & $\splM ^{1,\infty}$ &  $\splM ^{1,\infty}$  
\\[2ex]
     \textbf{15} &  $\splM ^{2,1}\wpr \splM ^{1,\infty}$
     & $\splM ^{1,\infty}$ &  --- 
\\[2ex]
     \textbf{16} &  $\splM ^{2,1}\wpr \splM ^{2,1}$ & $\splM ^{1,1}$ &  $\splM ^{1,1}$  
\\[2ex]
     \textbf{17} &  $\splM ^{2,1}\wpr \splM ^{2,2}$ & $\splM ^{1,2}$
     &  $\splM ^{1,2}$  
\\[2ex]
     \textbf{18} &  $\splM ^{2,1}\wpr \splM ^{2,\infty}$
     & $\splM ^{1,\infty}$ &  $\splM ^{1,\infty}$  
\\[2ex]
     \textbf{19} &  $\splM ^{2,1}\wpr \splM ^{\infty ,2}$
     & $\splM ^{2,2}$ &  $\splM ^{2,2}$  
\\[2ex]
     \textbf{20} &  $\splM ^{2,1}\wpr \splM ^{\infty ,\infty}$
     & $\splM ^{2,\infty}$ &  $\splM ^{2,\infty}$  
\\[2ex]
     \textbf{21} &  $\splM ^{2,2}\wpr \splM ^{\infty ,2}$
     & $\splM ^{2,\infty}$ &  $\splM ^{2,\infty}$  
\\[2ex]
     \textbf{22} &  $\splM ^{\infty ,2}\wpr \splM ^{\infty ,2}$ & $\splM ^{\infty ,\infty}$ &
     $\splM ^{\infty ,\infty}$ 
\end{tabular}}

\par

%%%%%%%%%%%%%%%%%%%%%%%%%%%%%


\begin{thebibliography}{2000}
%%%%%%%%%%%%%%%%%%%%%%%%%%%%%


\bibitem{BeLo} {J.~Bergh and J.~L{\"o}fstr{\"o}m},
\emph{Interpolation Spaces, An Introduction}, Springer-Verlag, {Berlin
Heidelberg New York}, 1976.

\bibitem{Cordero} E.~Cordero, \emph{Gelfand--Shilov window
classes for weighted modulation spaces}, Integr. Transf. Spec.
Funct. \textbf{18} (2007),  829--837.

\bibitem{CPRT10} E.~Cordero, S.~Pilipovi\'c, L.~Rodino and N.~Teofanov, 
\emph{Quasianalytic Gelfand--Shilov spaces with applications
to localization operators}, Rocky Mount. J. Math. \textbf{40} (2010),
1123-1147.

\bibitem{CTW}
E.~Cordero, A.~Tabacco and P.~Wahlberg, \emph{Schr{\"o}dinger
type propagators, pseudodifferential operators and modulation
spaces}, J. London Math. Soc., \textbf{88} (2) (2013),  375--395. 

\bibitem{Feichtinger1} H.~G.~Feichtinger,  \emph{Banach convolution algebras of
Wiener's type, {\rm in: Proc. Functions, Series, Operators in
Budapest}}, {Colloquia Math. Soc. J. Bolyai, North Holland
Publ. Co.}, Amsterdam Oxford NewYork, 1980.

\bibitem{Feichtinger2} H.~G.~Feichtinger, \emph{Modulation spaces on
locally compact abelian groups. Technical report}, {University of
Vienna}, Vienna, 1983; also in: M. Krishna, R. Radha,
S. Thangavelu (Eds), Wavelets and their applications, Allied
Publishers Private Limited, NewDehli Mumbai Kolkata Chennai Hagpur
Ahmedabad Bangalore Hyderbad Lucknow, 2003, pp. 99--140.

\bibitem{Feichtinger6} {H.~G.~Feichtinger},
\emph{Modulation spaces: looking back and ahead}, {Sampling
Theory in Signal and Image Processing} \textbf {5} (2) (2006), 109--140.

\bibitem{Feichtinger3}  {H.~G.~Feichtinger and K.~Gr{\"o}chenig},
\emph{Banach spaces related to integrable group representations and
their atomic decompositions, I}, J. Funct. Anal. \textbf{86}
(1989), 307--340.

\bibitem{Feichtinger4} {H.~G.~Feichtinger and K.~Gr{\"o}chenig},
\emph{Banach spaces related to
integrable group representations and their atomic decompositions, II},
Monatsh. Math. \textbf{108} (1989), 129--148.

\bibitem{Feichtinger5} {H.~G.~Feichtinger and K.~Gr{\"o}chenig},
\emph{Gabor frames and time-frequency analysis of distributions}, {J.
Funct. Anal.} \textbf {146} (2) (1997), 464--495.

\bibitem{Folland}
{G.~B.~Folland}, \emph{Harmonic Analysis in Phase Space}, Annals of
Mathematics Studies \textbf{122}, Princeton University Press, Princeton NJ, 1989. 

\bibitem{GS} {I.~M.~Gel'fand and G.~E.~Shilov}, \emph{Generalized functions, I--IV},
Academic Press, NewYork London, 1968.

\bibitem{GraciaVarilly}
J.~M.~Gracia-Bond\'ia and J.~C.~V\'arilly, 
\emph{Algebras of distributions suitable for phase-space quantum mechanics. I--II}, 
J. Math. Phys. \textbf{29} (4) (1988) 869--887.

\bibitem{Grobner}  {P.~Gr{\"o}bner}, \emph{Banachr{\"a}ume Glatter
Funktionen und Zerlegungsmethoden}, (Thesis), University
of Vienna, Vienna, 1992.

\bibitem{Grochenig0a} {K.~Gr{\"o}chenig}, \emph{Describing functions:
atomic decompositions versus frames}, {Monatsh. Math.}\textbf{112}
(1991), 1--42.

\bibitem{Grochenig2} {K.~Gr{\"o}chenig}, \emph{Foundations of
Time-Frequency Analysis}, {Birkh{\"a}user}, Boston, 2001.

\bibitem{Grochenig3} {K.~Gr{\"o}chenig}, \emph{Composition and
spectral invariance of pseudodifferential operators on modulation
spaces}, J. Anal. Math. \textbf{98} (2006), 65--82.

\bibitem{Grochenig4} {K.~Gr{\"o}chenig}, \emph{Time-frequency
analysis of Sj{\"o}strand's class}, Rev. Mat. Iberoam. \textbf{22}
(2006), 703--724.

\bibitem{Grochenig5} {K.~Gr{\"o}chenig}, \emph{Weight functions
in time-frequency analysis} {in: L. Rodino, B. W. Schulze, M. W.
Wong (Eds)} Pseudo-differential operators: Partial differential
equations and time-frequency analysis, Fields Institute
Communications \textbf{52}, AMS, Providence, RI, 2007,
pp. 343--366.

\bibitem{Grochenig0}  {K.~Gr{\"o}chenig and C.~Heil}, \emph
{Modulation spaces and pseudo-differential operators}, Integr.
Equ. Oper. Theory \textbf{34} (4) (1999), 439--457.

\bibitem{Grochenig1b} {K.~Gr{\"o}chenig and C.~Heil},
\emph{Modulation spaces as symbol
classes for pseudodifferential operators {\rm {in: M. Krishna,
R. Radha, S. Thangavelu (Eds), Wavelets and their applications}}},
Allied Publishers Private Limited, NewDehli Mumbai Kolkata Chennai
Hagpur Ahmedabad Bangalore Hyderbad Lucknow, 2003, pp. 151--170.

\bibitem{Grochenig1c} K.~Gr{\"o}chenig and C.~Heil,
\emph{Counterexamples for boundedness of
pseudodifferential operators}, Osaka  J. Math. \textbf{41} (2004),
681--691.

\bibitem{GrochenigRzeszotnik}
{K.~Gr\"ochenig and Z.~Rzeszotnik}, \emph{Banach algebras of
pseudodifferential operators and their almost diagonalization},
Ann. Inst. Fourier \textbf{58} (7) (2008),  2279--2314. 

\bibitem{GrochToft1} K.~Gr{\"o}chenig and J.~Toft, \emph{Isomorphism
properties of Toeplitz operators and pseudo-differential operators
between modulation spaces}, J. Anal. Math. \textbf{114} (2011),
255--283.

\bibitem{GZ} K.~Gr{\"o}chenig and G.~Zimmermann, \emph{Spaces of
test functions via the STFT}, J. Funct. Spaces Appl. \textbf{2} (2004),
25--53.

\bibitem{HanWang}
J.~ Han and B.~Wang, \emph{$\alpha$-modulation spaces (I): scaling,
embedding and algebraic properties}, 
arXiv:1108.0460v4 [math.FA], 2012. 

\bibitem{HTW}
A.~Holst, J.~Toft and P.~Wahlberg, \emph{Weyl product algebras and
modulation spaces}, J. Funct. Anal. \textbf{251} (2007), 463--491.

\bibitem{Ho3}  L.~H{\"o}rmander, \emph{The Analysis of Linear
Partial Differential Operators}, vol {I, III},
{Springer-Verlag}, Berlin Heidelberg NewYork Tokyo, 1983, 1985.

\bibitem{Ko} H.~Komatsu, \emph{Ultradistributions. I},
J. Fac. Sci. Univ. Tokyo Sect. IA Math. \textbf{20} (1973), 25--205.

\bibitem{Labate1} D. Labate, \emph{Time-frequency analysis of
pseudodifferential operators}, Monatshefte Math. \textbf{133} (2001),
143--156.

\bibitem{LozPer} 
Z.~Lozanov-Crvenkovi{\'c} and D.~Peri{\v{s}}i{\'c},
\emph{Kernel theorems for the spaces of tempered ultradistributions}, 
Integr. Transf. Spec. Funct. \textbf{18} (10) (2007), 699--713. 

\bibitem{Okoudjou} K.~Okoudjou, \emph{Embedding of some classical
Banach spaces into modulation spaces}, Proc. Amer. Math. Soc. \textbf{132}
(2004), 1639--1647.

\bibitem{Pil} S.~Pilipovi{\'c}, \emph{Tempered ultradistributions},
Boll. Un. Mat. Ital. B (7) \textbf{2} (1988), 235--251

\bibitem{Pilipovic1} {S.~Pilipovi\'c and N.~Teofanov}, \emph {Wilson bases
and ultramodulation spaces}, {Math. Nachr.} \textbf {242} (2002),
179--196.

\bibitem{Pilipovic2} {S.~Pilipovi\'c and N.~Teofanov}, \emph{On a symbol
class of elliptic pseudodifferential operators}, {Bull. Acad. Serbe
Sci. Arts} \textbf {27} (2002), 57--68.

\bibitem{Sjo1}  {J.~Sj{\"o}strand}, \emph{An algebra of
pseudodifferential operators}, {Math. Res. L.} \textbf 1 (1994),
185--192.

\bibitem{Sjo2} {J. ~Sj{\"o}strand}, \emph{Wiener type algebras of
pseudodifferential operators}, S\'eminaire Equations aux D\'eriv\'ees
Partielles, Ecole Polytechnique, 1994/1995, {Expos\'e n$^{\circ}$ IV.}

\bibitem{Sugimoto1} {M.~Sugimoto and N.~Tomita}, \emph{The dilation
property of modulation spaces and their inclusion relation with Besov
Spaces}, {J. Funct. Anal.} \textbf{248} (1)  (2007),
79--106.

\bibitem{Tachizawa1}  {K.~Tachizawa}, \emph{The boundedness of
pseudo-differential operators on modulation spaces},
Math. Nachr. \textbf{168} (1994), 263--277.

\bibitem{Teo1}  {N.~Teofanov}, \emph{Ultramodulation spaces and
pseudodifferential operators}, {Endowment Andrejevi\'c}, Beograd,
2003.

\bibitem{Teo2} {N.~Teofanov}, \emph{Modulation spaces, Gelfand--Shilov
spaces and pseudodifferential operators}, Sampl. Theory Signal
Image Process., \textbf{5} (2006), 225--242.

\bibitem{Toft0} {J.~Toft} \emph{Subalgebras to a Wiener type algebra of
pseudo-differential operators}, {Ann. Inst. Fourier} \textbf {51} (5)
(2001), 1347--1383.

\bibitem{Toft1} {J.~Toft}, \emph{Continuity properties for
non-commutative convolution algebras with applications in
pseudo-differential calculus}, {Bull. Sci. Math.} \textbf {126} (2)
(2002), 115--142.

\bibitem{Toft1.5} {J.~Toft}, \emph{Positivity properties for
non-commutative convolution algebras with applications in
pseudo-differential calculus}, {Bull. Sci. Math.} \textbf {127} (2)
(2003), 101--132.

\bibitem{Toft2} J.~Toft ,\emph{Continuity properties for
modulation spaces with applications to pseudo-differential calculus,
I}, {J. Funct. Anal.}, \textbf{207}  (2) (2004),
399--429.

\bibitem{Toft3} J.~Toft, \emph{Continuity
properties for modulation spaces with applications to
pseudo-differential calculus, II}, {Ann. Global Anal. Geom.},
\textbf{26} (2004), 73--106.

\bibitem{Toft4} J.~Toft, \emph{Convolution and embeddings for
weighted modulation spaces {\rm {in: P. Boggiatto, R. Ashino,
M. W. Wong (Eds)}}, Advances in Pseudo-Differential Operators,}
Operator Theory: Advances and Applications \textbf{155},
Birkh{\"a}user Verlag, Basel 2004, pp. 165--186.

\bibitem{Toft5} {J.~Toft}, \emph{Continuity and Schatten
properties for pseudo-differential operators on modulation spaces {\rm
{in: J. Toft, M. W. Wong, H. Zhu (eds)}} Modern Trends in
Pseudo-Differential Operators,} Operator Theory: Advances and
Applications, Birkh{\"a}user Verlag, Basel, 2007, 173--206.

\bibitem{Toft8} J.~Toft, \emph{The Bargmann transform on modulation and
Gelfand--Shilov spaces, with applications to Toeplitz and
pseudo-differential operators}, J. Pseudo-Differ. Oper. Appl.
\textbf{3} (2012), 145--227.

\bibitem{Toft9} J.~Toft, \emph{Multiplication properties in Gelfand--Shilov
pseudo-differential calculus {\rm {in: S. Molahajlo, S. Pilipovi{\'c}, J. Toft,
M. W. Wong (Eds)}}, Pseudo-Differential Operators, Generalized
Functions and  Asymptotics,} Operator Theory: Advances and
Applications \textbf{231}, Birkh{\"a}user, Basel Heidelberg NewYork
Dordrecht London, 2013, pp. 117--172.

\bibitem{TW}
J.~Toft and P.~Wahlberg, \emph{Embeddings of $\alpha$-modulation spaces}, 
Pliska Stud. Math. Bulgar. \textbf{21} (2012), 25--46.

\bibitem{WH} B.~Wang and C.~Huang, \emph{Frequency-uniform decomposition
method for the generalized BO, KdV and NLS equations}, J. Differential
Equations \textbf{239} (2007), 213--250.

\end{thebibliography}
\end{document}